\documentclass[11pt, leqno]{amsart}
\usepackage{a4wide}

\usepackage[utf8]{inputenc}
\usepackage{amssymb,amsmath,amsthm}
\usepackage{tikz} 
\usepackage{bm} 
\usepackage{mathrsfs}
\usepackage{enumitem} 

\usepackage[alphabetic, initials]{amsrefs}
\usepackage{amsrefs}

\numberwithin{equation}{section}

\theoremstyle{plain}
\newtheorem{theorem}{Theorem}[section]

\newtheorem{proposition}[theorem]{Proposition}
\newtheorem{lemma}[theorem]{Lemma}

\newtheorem{example}[theorem]{Example}

\newtheorem*{theorem*}{Theorem}
\newtheorem*{lemma*}{Lemma}
\newtheorem*{proposition*}{Proposition}
\newtheorem*{corollary*}{Corollary}

\theoremstyle{definition}
\newtheorem{definition}[theorem]{Definition}

\newcommand{\bR}{\mathbb R}

\title{Diophantine approximation by rational numbers of certain parity types}

\author{Dong Han Kim}
\address{Department of Mathematics Education, 
Dongguk University - Seoul, Seoul, 04620 Korea}
\email{kim2010@dgu.ac.kr}

\author{Seul Bee Lee}
\address{Center for Geometry and Physics, Institute for Basic Science (IBS), Pohang 37673, Korea}
\email{seulbee.lee@ibs.re.kr}

\author{Lingmin Liao}
\address{School of Mathematics and Statistics, Wuhan University, Wuhan 430072, China}
\email{lmliao@whu.edu.cn}

\date{\today}

\begin{document}
\maketitle

\begin{abstract}
For a given irrational number, we consider the properties of best rational approximations of given parities.
There are three different kinds of rational numbers according to the parity of the numerator and denominator, say odd/odd, even/odd and odd/even rational numbers.
We study algorithms to find best approximations by rational numbers of given parities and compare these algorithms with continued fraction expansions.
\end{abstract}

\section{Introduction}
A fundamental problem in Diophantine approximation is to find rational numbers $p/q$ minimizing $|q x-p|$ for $x \in\bR \setminus \mathbb{Q}$ with some constraints of $p$ and $q$.
A rational $p/q$ in lowest terms is called a \emph{best approximation} of $x$ if any rational $a/b$ in lowest terms apart from $p/q$ such that $0 < b \le q$ satisfies
$$|qx-p|<|bx-a|.$$
We denote by $\mathscr{B}(x)$ the set of all best approximations of $x$.
One significant property of the regular continued fraction is that a rational number is a principal convergent of $x$ if and only if it is a best approximation of $x$ or it is $\lfloor x \rfloor$, which was first proven by Lagrange, see \cite{Khi64}*{Section 6}.
Every real number $x$ has a continued fraction expansion 
$$ 
x=a_0+\cfrac{1}{a_1+\cfrac{1}{a_2+\ddots}}=:[a_0;a_1,a_2,\cdots],
$$
where $a_0=\lfloor x \rfloor\in\mathbb Z$, and $a_n\in\mathbb N$ for $n\ge 1$.
Let $p_n/q_n=[a_0;a_1,a_2,\cdots,a_n]$ be a principal convergent
of $x \in \mathbb R \setminus \mathbb{Q}$ for $n\ge 0$ with the convention $p_{-1} = 1$, $q_{-1} = 0$  and  $p_0 = a_0$, $q_0 = 1$.
By Lagrange's theorem, we have
\begin{equation}\label{eq:P=B}
\mathscr{B}(x) = 
\begin{cases}
\left\{ \frac{p_n}{q_n} : n \ge 0 \right\}, & \text{ if } a_1 \ge 2, 
\\
\left\{ \frac{p_n}{q_n} : n \ge 1 \right\}, &\text{ if } 
a_1 =1. 
\end{cases}
\end{equation}

For \emph{the even continued fraction}, which allows only positive even integer partial quotients  \cites{Sch82,Sch84}, an analogous theorem holds, i.e., the principal convergents of the even continued fraction are the best approximations among the rationals of the form even/odd and odd/even, and vice versa \cite{SW14}, see also \cite{KL96}.
In \cite{KLL22}, the authors defined a continued fraction algorithm, called the \emph{odd-odd continued fraction}, giving the best approximations of the form odd/odd.
The aim of the present paper is to further investigate the best approximations of specific parity and the associated continued fraction algorithms.

We classify the set of rational numbers by the following three subsets with the parity conditions:
\begin{align*}
\mathbb Q^{(0)} &:= \left\{ \frac pq \in \mathbb Q \ : \ p \equiv 0,~ q \equiv 1 \pmod 2 \right\}, \\
\mathbb Q^{(1)} &:= \left\{ \frac pq \in \mathbb Q \ : \ p \equiv 1,~ q \equiv 1 \pmod 2 \right\}, \\
\mathbb Q^{(\infty)} &:= \left\{ \frac pq \in \mathbb Q \ : \ p \equiv 1,~ q \equiv 0 \pmod 2 \right\}.
\end{align*}
We call an element of $\mathbb Q^{(0)}$ (respectively $\mathbb Q^{(1)}$, $\mathbb Q^{(\infty)}$) \emph{a $(0)$-rational} (respectively \emph{a $(1)$-rational, an $(\infty)$-rational}) number.
We also denote 
$$
\mathbb Q^{(0,1)} = \mathbb Q^{(0)} \sqcup \mathbb Q^{(1)}, \quad
\mathbb Q^{(0,\infty)} = \mathbb Q^{(0)} \sqcup \mathbb Q^{(\infty)}, \quad
\mathbb Q^{(1,\infty)} = \mathbb Q^{(1)} \sqcup \mathbb Q^{(\infty)}.
$$
From now on, we indicate one of $0,1,\infty$ with a symbol $\alpha,\beta,\gamma$ for brevity.
We define a best approximation among a specific class $\mathbb Q^{(\alpha)}$ or $\mathbb Q^{(\beta,\gamma)}$ for $\beta\neq \gamma$.
\begin{definition}\label{de:(a)BA}
For given $x \in \mathbb R\setminus\mathbb Q$, we define
\emph{a best $(\alpha)$-rational (respectively $(\beta,\gamma)$-rational) approximation} by a rational $p/q$ satisfying
$$\left| qx- p \right|<\left| bx- a \right|$$   
for any $a/b \in \mathbb Q^{(\alpha)}$ (respectively $\mathbb{Q}^{(\beta,\gamma)}$) apart from $p/q$ such that $ 0< b \le q$.
\end{definition}

Denote by $\mathscr{B}^{(\alpha)}$ (respectively $\mathscr{B}^{(\beta,\gamma)}$) the set of the best $(\alpha)$-rational (respectively $(\beta,\gamma)$-rational) approximations.
By definition, it is deduced that for any $\alpha, \beta \in \{ 0, 1, \infty\}$ with $\alpha \ne \beta$, we have
\begin{equation}\label{def_cor1}
\mathscr{B}(x) \cap \mathbb Q^{(\alpha)}\subset \mathscr{B}^{(\alpha)}(x),\quad
\mathscr{B}(x) \cap \mathbb Q^{(\alpha,\beta)} \subset \mathscr{B}^{(\alpha, \beta)}(x), \quad
\mathscr{B}^{(\alpha, \beta)}(x) \cap \mathbb Q^{(\alpha)} \subset \mathscr{B}^{(\alpha)}(x).
\end{equation}
Therefore, we have
\begin{equation}\label{def_cor2}
\mathscr{B}^{(\alpha, \beta)}(x) 
= \left( \mathscr{B}^{(\alpha, \beta)}(x) \cap \mathbb Q^{(\alpha)} \right) 
\sqcup 
\left( \mathscr{B}^{(\alpha, \beta)}(x) \cap \mathbb Q^{(\beta)} \right)
\subset 
\mathscr{B}^{(\alpha)}(x) \sqcup \mathscr{B}^{(\beta)}(x).
\end{equation}

We call $p/q$ \emph{a best signed approximation of $x\in\mathbb R$} 
if any rational $a/b$ in lowest terms apart from $p/q$ such that $0 < b \le q$ and $\mathrm{sgn}(bx-a)=\mathrm{sgn}(qx-p)$
satisfies 
$$bx-a  < qx-p < 0 \quad \text{ or } \quad
 0 < qx-p < bx-a$$
depending on the sign of $qx-p$. 
We denote by $\mathscr{S}(x)$ the set of the best signed approximations of $x$.
Obviously, we have $\mathscr{B}(x) \subset \mathscr{S}(x)$.
We decompose $\mathscr{S}(x)\setminus \mathscr{B}(x)$ by the following definition.
\begin{definition}\label{de:S_alpha}
We write $p/q \in \mathscr{S}_\alpha (x)$ if $p/q \in \mathscr{S}(x) \setminus \mathscr{B}(x)$ and there exists $a/b \in \mathbb Q^{(\alpha)}$ not equal to $p/q$ such that $0 < b \le q$ and $|bx - a | \le | qx - p|$.
\end{definition}
It is clear that 
\begin{equation}\label{mmm}
\mathscr{S}(x) \setminus \mathscr{B}(x) = \mathscr{S}_0(x) \cup \mathscr{S}_1(x) \cup \mathscr{S}_\infty (x).
\end{equation}
Moreover, the unions are disjoint. 
See Proposition~\ref{Salpha} below for the detail.

We first characterize the set of best $(\alpha)$-rational approximations and the set of best $(\alpha,\beta)$-rational approximations using best signed approximations.

\begin{theorem}\label{thm1}
Let $x$ be an irrational number. 
\begin{enumerate}[label=(\roman*)]
\item\label{it:thm1-1} For any $\alpha \in \{ 0, 1, \infty\}$, we have
$$
\mathscr{B}^{(\alpha)}(x) = \mathscr{S} (x) \cap \mathbb Q^{(\alpha)}.
$$

\item
For any $\{\alpha, \beta, \gamma\} = \{ 0, 1, \infty\}$, 
we have
$$
\mathscr{B}^{(\alpha,\beta)} (x) = \left( \mathscr{B} (x) \cap \mathbb{Q}^{(\alpha,\beta)} \right) \sqcup \mathscr{S}_\gamma (x).
$$
\end{enumerate}
\end{theorem}
As a corollary, we have
$$
\mathscr{S}(x) = \mathscr{B}^{(0)} (x) \sqcup \mathscr{B}^{(1)} (x) \sqcup \mathscr{B}^{(\infty)} (x) = \mathscr{B}^{(0,1)} (x) \cup \mathscr{B}^{(0,\infty)} (x) \cup \mathscr{B}^{(1,\infty)} (x)
$$
and 
\begin{equation*}
\mathscr{B}^{(\alpha,\beta)} (x) \cap \mathscr{B}^{(\alpha,\gamma)} (x) = \mathscr{B} (x) \cap \mathbb{Q}^{(\alpha)}.
\end{equation*}
The proof of Theorem~\ref{thm1} is given in Section~\ref{sec2}.
We remark that for any $\alpha, \beta \in \{ 0, 1, \infty\}$ with $\alpha \ne \beta$, we have, by \eqref{def_cor2},
$$
\mathscr{B}^{(\alpha,\beta)} (x) \subset \mathscr{B}^{(\alpha)}(x) \sqcup \mathscr{B}^{(\beta)}(x)
= \mathscr{S}(x) \cap \mathbb Q^{(\alpha,\beta)}.
$$
However, the equality does not hold in general.
For an example, let $x = \sqrt 2 -1$, then $\frac 37 \in \mathscr{S}(x) \cap \mathbb Q^{(0,1)}$   (see Example~\ref{ex:bs}). 
However, since $\frac 25 \in \mathbb Q^{(0,1)}$ satisfies $|5x -2| < |7x-3|$, we deduce that $\frac 37 \notin \mathscr{B}^{(0,1)} (x)$. 
In Section~\ref{sec:CF}, we prove the equivalence between the intermediate convergents and the signed best approximation, and we specify which intermediate convergent will be a best ($\alpha$)-rational approximation and a best ($\alpha$, $\beta$)-rational approximation.

In the rest of the article, we investigate algorithms to find best approximations along the specific parities.
We use the reflections 
$$
H_0 := \begin{pmatrix} -1 & 2 \\ 0 & 1 \end{pmatrix}, \qquad
H_1 := \begin{pmatrix} -1 & 0 \\ 0 & 1 \end{pmatrix}, \qquad
H_\infty := \begin{pmatrix} 1 & 0 \\ 2 & -1 \end{pmatrix}.
$$
on the upper half plane $\mathbb H$ which preserve the parities of the numerator and denominator of fractions on the boundary.
The relation between the continued fraction and cutting sequence of the geodesic on the upper half plane $\mathbb H$ is well studied. 
See \cites{Ser85, Ser91, EW11} for the regular continued fraction, \cite{BM18} for the even continued fraction.
We consider a symbolic sequence of the geodesic 
associated with the triangle group $\Delta$ generated by $H_0, H_1, H_\infty$ in Section~\ref{sec:Delta}.
Using this symbolic sequence, we give  algorithms generating best approximations with the parity condition.

\begin{theorem}\label{thm2}
Let $x = [ \alpha_1, \alpha_2, \dots ]_\Delta$ be the $\Delta$-expression of $x$ which will be defined in Section~\ref{sec:Delta}.
Then we have 
\begin{align*}
\mathscr{B}(x) &= \{ [ \alpha_1, \dots, \alpha_{m-1}, \overline{ \alpha_m, \alpha_{m+1}}]_\Delta \in \mathbb Q \, | \, m \ge | a_0| +1 \}, \\
\mathscr{S}(x) &= \{ [\alpha_1, \dots, \alpha_{m-1}, \overline{ \alpha_m, \delta_{m+1}}]_\Delta \in \mathbb Q \, | \, m \ge |a_0| +1 \}, \\
\mathscr{B}^{(\alpha)}(x) &= \{ [ \alpha_1, \dots, \alpha_{m-1}, \overline{ \alpha_m, \delta_{m+1}}]_\Delta \in \mathbb Q \, | \, \alpha_{m+1} = \alpha, \ m \ge |a_0| +1 \}, \\
\mathscr{B}^{(\alpha,\beta)}(x) &= \{ [ \alpha_1, \dots, \alpha_{m-1}, \overline{ \alpha_m, \gamma }]_\Delta \in \mathbb Q \, | \, \alpha_m \ne \gamma, \ m \ge |a_0| +1 \}.
\end{align*}
where $\delta_{m+1}$ satisfies $\{ \alpha_{m}, \alpha_{m+1}, \delta_{m+1} \} = \{ 0 , 1 , \infty \}$ for $m \ge 1$
and $\{\alpha, \beta, \gamma\} = \{ 0,1,\infty\}$.
\end{theorem}
The proof of Theorem~\ref{thm2} is given in Section~\ref{sec:BA} by Propositions \ref{BA}, \ref{aBA} and \ref{abBA}.
In Section~\ref{sec:maps}, we explain the relation between various continued fraction maps and the $\Delta$-expression. 

\section{Best approximations of the specified parity}
\label{sec2}

In this section, we characterise the set of best $(\alpha)$- and $(\beta,\gamma)$-rational approximations and prove Theorem~\ref{thm1}.

Let $\mathbf v = (p,q)$ be a vector in $\mathbb Z^2$ and $x$ be an irrational number.
Denote by $P_{\mathscr S} (x,\mathbf v)$ the parallelogram given by two vectors $\mathbf x(x,\mathbf v) = (p-qx,0)$ and $\mathbf y(x,\mathbf v) = (qx,q)$ with the boundary, i.e.,
$$
P_{\mathscr S}(x, \mathbf v) = \{ a \mathbf x(x,\mathbf v) + b \mathbf y(x,\mathbf v) \in \mathbb R^2 \, | \  0 \le a,b \le 1 \}.
$$
We also define
$$
P_{\mathscr B}(x, \mathbf v) = \{ a \mathbf x(x,\mathbf v) + b \mathbf y(x,\mathbf v) \in \mathbb R^2 \, | \  -1 \le a,b \le 1 \}.
$$
See Figure~\ref{Fig_parallelograms}.
For a subset $P$ of $\mathbb R^2$, we denote by $-P=\{-\mathbf v: \mathbf v\in P\}$.

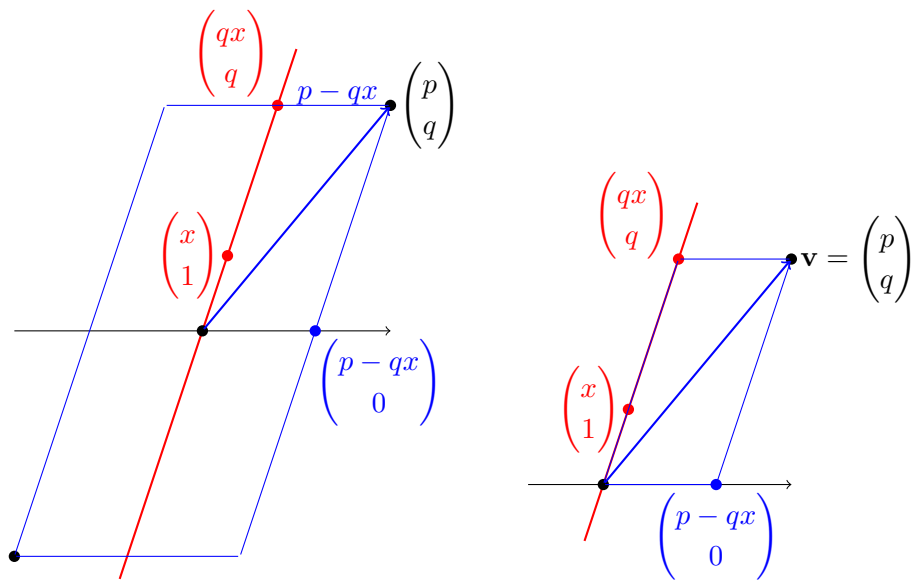
\begin{figure}
\begin{tikzpicture}[scale=0.5,inner sep=0mm]
	\draw[->] (-5, 0) -- (5, 0);
	\draw[thick,red] (-2.2, -6.6) -- (2.5, 7.5);
	\foreach \position in {
	(0, 0), (-5, -6), (5, 6)}
    \fill \position circle (0.15cm);
	\foreach \position in {
	(.66667, 2), (2,6)}
    \fill[red] \position circle (0.15cm);
    \foreach \position in {(3,0)}
    \fill[blue] \position circle (0.15cm);
	\node (O) at (0, 0) {};
	\node (P) at (5, 6) {};
	\node (P1) at (-1, 6) {};
	\node (P2) at (1,-6) {};
	\node (P3) at (-5,-6) {};
	\node[red, left, xshift=-3pt] (Q) at (.66667 ,2) {$\begin{pmatrix} x \\ 1 \end{pmatrix}$};
	\node[red, above left, xshift=-3pt, yshift=3pt] (Q) at (2 ,6) {$\begin{pmatrix} qx \\ q \end{pmatrix}$};
 	\node[blue, below right, yshift=-3pt] (Q) at (3 ,0) {$\begin{pmatrix} p-qx \\ 0 \end{pmatrix}$};
	\draw[thick, blue, ->] (O) -- (P) node[black, right, xshift=3pt]{$\begin{pmatrix} p \\ q \end{pmatrix}$};
	\draw[thin, blue] (P) -- (P1) node[above right,xshift=50pt] {$p-qx$};
	\draw[thin, blue] (P) -- (P2);
	\draw[thin, blue] (P3) -- (P1);
	\draw[thin, blue] (P3) -- (P2);
\end{tikzpicture}
\qquad
\begin{tikzpicture}[scale=0.5,inner sep=0mm]
	\draw[->] (-2, 0) -- (5, 0);
	\draw[thick,red] (-.5, -1.5) -- (2.5, 7.5);
	\foreach \position in {
	(0, 0), (5, 6)}
    \fill \position circle (0.15cm);
	\foreach \position in {
	(.66667, 2), (2,6)}
    \fill[red] \position circle (0.15cm);
     \foreach \position in {(3,0)}
    \fill[blue] \position circle (0.15cm);
	\node (O) at (0, 0) {};
	\node (P) at (5, 6) {};
	\node (P1) at (2, 6) {};
	\node (P2) at (3,0) {};
	\node (P3) at (0,0) {};
	\node[red, left, xshift=-3pt] (Q) at (.66667, 2) {$\begin{pmatrix} x \\ 1 \end{pmatrix}$};
	\node[red, above left, xshift=-3pt] (Q) at (2 ,6) {$\begin{pmatrix} qx \\ q \end{pmatrix}$};
	\node[blue, below, yshift=-3pt] (Q) at (3 ,0) {$\begin{pmatrix} p-qx \\ 0 \end{pmatrix}$};
	\draw[thick, blue, ->] (O) -- (P) node[black, right, xshift=3pt]{$\mathbf v = \begin{pmatrix} p \\ q \end{pmatrix}$};
	\draw[thin, blue] (P) -- (P1);
	\draw[thin, blue] (P) -- (P2);
	\draw[thin, blue] (P3) -- (P1);
	\draw[thin, blue] (P3) -- (P2);
\end{tikzpicture}
\caption{The parallelograms 
$P_{\mathscr B}(x,\mathbf v_{p/q})$ (left) and $P_{\mathscr S}(x,\mathbf v_{p/q})$ (right)}
\label{Fig_parallelograms}
\end{figure}

We call a nonzero integral vector $\mathbf v \in \mathbb Z^2$ \emph{primitive} if $\mathbf v \ne c\mathbf w $ for any $c \ge 2$, $w \in \mathbb Z^2$.
For a rational number $p/q$, let $\mathbf v_{p/q} = (p,q)$. 
A rational number $p/q$ is a best approximation of $x$ if and only if there are no primitive vectors of $\mathbb Z^2$ in $P_{\mathscr B} (x,\mathbf v_{p/q})$ except for $\mathbf v_{p/q}$, $-\mathbf v_{p/q}$ and $\mathbf 0 = (0,0)$.
Similarly, a rational number $p/q$ is a best signed approximation of $x$ if and only if there are no vectors of $\mathbb Z^2$ in $P_{\mathscr S} (x,\mathbf v_{p/q})$ except for $\mathbf v_{p/q}$ and $\mathbf 0$.

Let
\begin{align*}
\Lambda^{(0)} &= \left \{ (p,q) \in \mathbb Z^2 \, | \, p \equiv 0 \pmod 2 \right\}, \\
\Lambda^{(1)} &= \left \{ (p,q) \in \mathbb Z^2 \, | \,p+q \equiv 0 \pmod 2 \right\}, \\
\Lambda^{(\infty)} &= \left \{ (p,q) \in \mathbb Z^2 \, | \, q \equiv 0 \pmod 2 \right\}
\end{align*}
be index 2-sublattices of $\mathbb Z^2$.
Let $p/q \in \mathbb Q^{(\alpha)}$.
A rational number $p/q$ is a best $(\alpha)$-rational approximation (respectively, best $(\alpha,\beta)$-rational approximation) of $x$ if and only if there is no primitive vector belongs to $\Lambda^{(\alpha)}$ (respectively, $\Lambda^{(\alpha)} \cup \Lambda^{(\beta)}$) in $P_{\mathscr B} (x,\mathbf v_{p/q})$ except for $\mathbf v_{p/q}$, $-\mathbf v_{p/q}$ and $\mathbf 0$.

\begin{lemma}\label{lattice}
Let $x \in \mathbb R \setminus \mathbb Q$. 
Let $\mathbf v \in \Lambda^{(\alpha)}$ be a nonzero primitive vector for some $\alpha \in \{ 0,1,\infty\}$.
\begin{enumerate}[label=(\roman*)]
\item If $P_{\mathscr B} (x,\mathbf v)$ contains no primitive vectors belonging to $\Lambda^{(\alpha)}$ except for $\mathbf v$, $-\mathbf v$, $\mathbf 0$, then $P_{\mathscr S} (x,\mathbf v)$ contains no vectors of $\mathbb Z^2$ except for $\mathbf v$, $\mathbf 0$.

\item 
Assume that $P_{\mathscr S} (x,\mathbf v)$ contains no vectors of $\mathbb Z^2$ except for $\mathbf v$ and $\mathbf 0$.
Then there exists no primitive vectors belonging to $\Lambda^{(\alpha)}$ in $P_{\mathscr B} (x,\mathbf v)$ except for $\mathbf v$, $-\mathbf v$, $\mathbf 0$.
Moreover, any two nonzero vectors $\mathbf u$, $\mathbf u'$ of $\mathbb Z^2$ in $P_{\mathscr B} (x,\mathbf v)$ besides of $\mathbf v$, $-\mathbf v$ satisfy that $\mathbf u = c\mathbf u'$ for some constant $c$.
\end{enumerate}
\end{lemma}

\begin{proof}
(i) Suppose that there exists a nonzero $\mathbf u  \in \mathbb Z^2$ in $P_{\mathscr S}(x,\mathbf v)$ with $\mathbf u \neq \mathbf v$.
Let 
$$
\mathbf w = \mathbf v - 2(\mathbf v -\mathbf u). 
$$
Then $\mathbf w \in \Lambda^{(\alpha)}$, $\mathbf w \ne \mathbf v$ and $\mathbf w \in P_{\mathscr B}(x,\mathbf v)$.
See Figure~\ref{Fig_proof} (left).
Let $c$ be the greatest common divisor of the two coordinates of $\mathbf w$. 
Since $\mathbf v$ is primitive, $c$ cannot be an even number. 
Thus $\frac 1c \mathbf w$ is primitive and $\frac 1c \mathbf w \in \Lambda^{(\alpha)}$.

(ii) 
Let ${\mathbf u}$ be a nonzero vector of $\mathbb Z^2$ in $P_{\mathscr B}(x,\mathbf v)$ which differs from $\mathbf v$, $-\mathbf v$.
By the assumption, $\mathbf u$ is not in $-P_{\mathscr S}(x,\mathbf v)$.
If ${\mathbf u} \in \Lambda^{(\alpha)}$, then 
$\mathbf w = \frac{1}{2}\left(\mathbf v + \mathbf u \right) \in \mathbb Z^2$ is in $P_{\mathscr S}(x,\mathbf v)$, which contradicts to the assumption.
See Figure~\ref{Fig_proof} (right).

Let $\mathbf u = a \mathbf x + b \mathbf y$ for some integers $a, b$ of opposite signs where $\mathbf x = \mathbf x(x,\mathbf v)$ and $\mathbf y = \mathbf y(x,\mathbf v)$.
We may assume that $P_{\mathscr B} (x,\mathbf u)$ contains no vectors of $\mathbb Z^2$ except for $\mathbf u$, $-\mathbf u$, $\mathbf 0$.
Then the parallelogram given by $\mathbf u$ and $\mathbf v$,
$$
P = \{ s\mathbf v + t \mathbf u \, | \, 0 \le s,t \le 1\} \subset P_{\mathscr S} (x,\mathbf v) \cup P_{\mathscr S} (x,\mathbf u)  \cup \left( P_{\mathscr S} (x,\mathbf v) + \mathbf u \right) \cup \left(  P_{\mathscr S} (x,\mathbf u) + \mathbf v \right),
$$
contains no integral vectors except for $\mathbf 0$, $\mathbf v$, $\mathbf u$, $\mathbf v + \mathbf u$. 
Since $\mathbf v$ and $\mathbf u$ generate $\mathbb Z^2$, there exist integers $s$, $t$ such that $\mathbf u' = s \mathbf v + t \mathbf u
= (s + a t) \mathbf x + (s + bt) \mathbf y$. 
By $\mathbf u' \notin P_{\mathscr S} (x,\mathbf v) \cup \left( - P_{\mathscr S} (x,\mathbf v) \right)$, we have  $-1 \le s + at < 0 < s + bt \le 1$ or $-1 \le s + bt < 0 < s + at \le 1$. 
Since $a, b$ are of opposite signs, we conclude that $s = 0$.
\end{proof}

\begin{figure}
\begin{tikzpicture}[scale=.6,inner sep=0mm]
	\draw[->] (-3.5, 0) -- (5, 0) node[right]{$x$};
	\draw[thick,red] (-.3, -.9) -- (2.3, 6.9);
	\foreach \position in {
	(0, 0), 
	(2,4), (-1,2),
	(5, 6)}
    \fill \position circle (0.15cm);
	\node (O) at (0, 0) {};
	\node (P) at (5, 6) {};
	\node (P1) at (-1, 6) {};
	\node (P2) at (2.7,-.9) {};
	\node (P3) at (-3.3,-.9) {};
	\node (P4) at (2,4) {};
	\node (P5) at (-1,2) {};
	
	\draw[thick, blue, ->] (O) -- (P) node[black, right, xshift=5pt]{$\mathbf v$};
	\draw[thick, blue, ->] (O) -- (P4) node[black, above, yshift = 3pt]{$\mathbf u$};
	\draw[thick, blue, ->] (O) -- (P5) node[black, above, yshift=3pt]{$\mathbf v -2(\mathbf v - \mathbf u)$};
	\draw[thin, dashed, blue] (P) -- (P1);
	\draw[thin, dashed, blue] (P) -- (P2);
	\draw[thin, dashed, blue] (P3) -- (P1);
\end{tikzpicture}
\qquad
\begin{tikzpicture}[scale=.6,inner sep=0mm]
	\draw[->] (-3.5, 0) -- (5, 0) node[right]{$x$};
	\draw[thick,red] (-.3, -.9) -- (2.3, 6.9);
	\foreach \position in {
	(0, 0), 
	(2,4), (-1,2),
	(5, 6)}
    \fill \position circle (0.15cm);
	\node (O) at (0, 0) {};
	\node (P) at (5, 6) {};
	\node (P1) at (-1, 6) {};
	\node (P2) at (2.7,-.9) {};
	\node (P3) at (-3.3,-.9) {};
	\node (P4) at (2,4) {};
	\node (P5) at (-1,2) {};
	
	\draw[thick, blue, ->] (O) -- (P) node[black, right, xshift=5pt]{$\mathbf v$};
	\draw[thick, blue, ->] (O) -- (P4) node[black, above, xshift=3pt,yshift = 3pt]{$\frac12(\mathbf v+\mathbf u)$};
	\draw[thick, blue, ->] (O) -- (P5) node[black, left, xshift=-3pt]{$\mathbf u$};
	\draw[thin, dashed, blue] (P) -- (P1);
	\draw[thin, dashed, blue] (P) -- (P2);
	\draw[thin, dashed, blue] (P3) -- (P1);
\end{tikzpicture}
\caption{Vectors of best approximations}
\label{Fig_proof}
\end{figure}
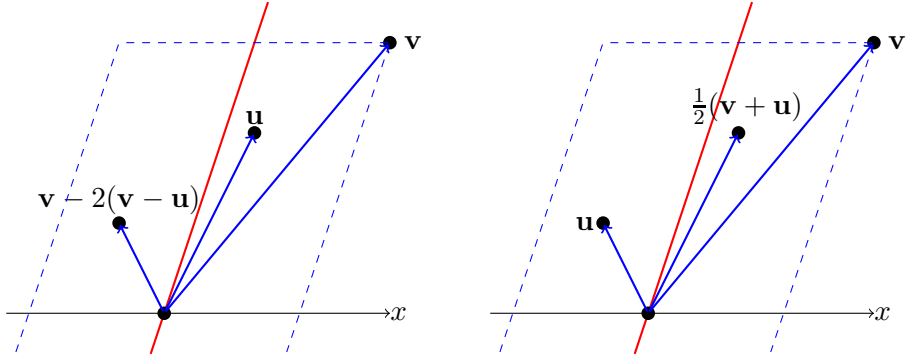

The following proposition is a direct consequence of Lemma~\ref{lattice}.
Note that $p/q \in \mathscr{S}_\alpha (x)$ if and only if $P_{\mathscr S} (x,\mathbf v_{p/q})$ contains no vectors of $\mathbb Z^2$ except for $\mathbf v_{p/q}$ and $\mathbf 0$, but $P_{\mathscr B} (x,\mathbf v_{p/q})$ contains a nonzero primitive vector of $\Lambda^{(\alpha)}$ except for $\mathbf v_{p/q}$.
By Lemma~\ref{lattice} (ii), we have $\mathscr{S}_\alpha(x)\cap \mathbb Q^{(\alpha)}=\emptyset$.

\begin{proposition}\label{Salpha}
Let $x \in \mathbb R \setminus \mathbb Q$.
For any $\alpha \in \{ 0, 1, \infty\}$, we have 
\begin{equation*}
\mathscr{B}^{(\alpha)}(x)\subset \mathscr{S}(x)\quad \text{ and } \quad
\mathscr{S}(x) \cap \mathbb Q^{(\alpha)} \subset \mathscr{B}^{(\alpha)}(x).
\end{equation*}
For $\alpha, \beta, \gamma$ with $\{ \alpha, \beta, \gamma \} = \{ 0, 1, \infty\}$, we have
\begin{equation}\label{eq:pr:Salpha-2}
\mathscr{S}_\alpha (x) \cap \mathscr{S}_\beta (x) = \emptyset, \qquad
\mathscr{S}_\alpha (x) \subset \mathscr{B}^{(\beta, \gamma)} (x), \qquad \text{and}\qquad
\mathscr{S}_\alpha(x) \cap \mathscr{B}^{(\alpha,\gamma)} (x) = \emptyset.
\end{equation}
\end{proposition}

Using Proposition~\ref{Salpha}, we prove Theorem~\ref{thm1} as follows.

\begin{proof}[Proof of Theorem~\ref{thm1}]
The first statement of Proposition~\ref{Salpha} implies \ref{it:thm1-1}.

From \eqref{def_cor1}, we have
$\mathscr{B} (x) \cap \mathbb{Q}^{(\alpha,\beta)} \subset \mathscr{B}^{(\alpha,\beta)} (x)$
and from \eqref{eq:pr:Salpha-2}, we have
$\mathscr{S}_\gamma (x) \subset \mathscr{B}^{(\alpha,\beta)} (x).$
Therefore, we deduce that 
$$
\left( \mathscr{B} (x) \cap \mathbb{Q}^{(\alpha,\beta)} \right) \sqcup  \mathscr{S}_\gamma(x) \subset \mathscr{B}^{(\alpha,\beta)} (x).
$$
For the other direction, by \eqref{def_cor2}, \eqref{mmm} and \ref{it:thm1-1}, we have
\begin{align*}
\mathscr{B}^{(\alpha,\beta)} (x) 
&\subset \mathscr{B}^{(\alpha)} (x) \sqcup \mathscr{B}^{(\beta)} (x) 
= \mathscr{S}(x) \cap \mathbb Q^{(\alpha,\beta)} \\
&= \left( \mathscr{B}(x) \cap \mathbb Q^{(\alpha,\beta)}  \right) \sqcup \left( (\mathscr{S}(x) \setminus \mathscr{B}(x) ) \cap \mathbb Q^{(\alpha,\beta)}  \right)\\
&= \left( \mathscr{B}(x) \cap \mathbb Q^{(\alpha,\beta)}  \right) \sqcup \left( \left( \mathscr{S}_\alpha(x) \cup
\mathscr{S}_\beta(x) \cup
\mathscr{S}_\gamma(x) \right) \cap \mathbb Q^{(\alpha,\beta)}  \right).
\end{align*}
By \eqref{eq:pr:Salpha-2}, we have 
$$ \left( \mathscr{S}_\alpha(x) \cup
\mathscr{S}_\beta(x) \cup \mathscr{S}_\gamma(x) \right) \cap \mathscr{B}^{(\alpha,\beta)} = \mathscr{S}_\gamma(x).$$
Therefore, we conclude that 
\begin{equation*}
\mathscr{B}^{(\alpha,\beta)} (x) 
\subset \left( \mathscr{B}(x) \cap \mathbb Q^{(\alpha,\beta)} \right) \sqcup \mathscr{S}_\gamma(x).  \qedhere
\end{equation*}
\end{proof}

\section{Best signed approximations and intermediate convergents}
\label{sec:CF}

In this section, we will show that best signed approximations are intermediate convergents. 
An \emph{intermediate convergent} (or semi-convergent or median convergent) is defined by
$$\frac{p_{n,k}}{q_{n,k}} = \frac{kp_{n-1}+p_{n-2}}{kq_{n-1}+q_{n-2}} \quad  \text{ for } \  1\leq  k < a_n, \ n\ge 1.$$
The following lemma states a condition for a rational being an intermediate convergent. 

\begin{lemma}\label{lem1}
Let $x \in\mathbb R \setminus \mathbb Q$.
If $\frac ab$ satisfies that
\begin{equation}\label{condition}
b < q_n\ \text{ and } \ 
| bx - a | < | q_{n-2} x - p_{n-2} | \qquad \text{ for } n \ge 1,
\end{equation}
then $\frac ab  = \frac{p_{n-1}}{q_{n-1}}$ or 
$\frac{p_{n,k}}{q_{n,k}}$ for some $1 \le k < a_n$.
\end{lemma}

\begin{proof}
Since $p_{n-1}/q_{n-1}$ is the only best approximation of $x$  with a denominator between $q_{n-2}$ and $q_n$, 
Condition \eqref{condition} implies that $b \ge q_{n-1}$ and $|q_{n-1} x - p_{n-1}| \le | bx- a |$.

We may assume that $n$ is odd, i.e., $q_{n-1} x - p_{n-1} > 0$ since the other case is symmetric.

\begin{enumerate}[label=(\roman*)]
\item If $0 < b x - a $, then 
$0 < q_{n-1} x - p_{n-1} \le bx- a$.
Therefore, $0 \le (b-q_{n-1}) x - (a - p_{n-1}) < b x- a$,
which implies either 
$b = q_{n-1}$ or $0 < (b-q_{n-1}) x - (a - p_{n-1}) < b x- a$.
The latter case implies $q_{n-1} \le b-q_{n-1}<q_n$ by \eqref{condition}, thus
$b-q_{n-1} = q_{n-1}$ or $0 < (b-2q_{n-1}) x- (a - 2p_{n-1}) < b x- a$. 
We repeat this procedure until $b = c q_{n-1}$ and $a = c p_{n-1}$ for some $c \ge 1$.

\item If $bx- a <0$, then we have
$q_{n-2} x - p_{n-2} < bx- a <0$.
Therefore, $ (b-q_{n-2}) x - (a-p_{n-2}) >0$. 
By (i), we have $b = c q_{n-1} + q_{n-2}$ for some $c$ with $1 \le c < a_n$. \qedhere
\end{enumerate}
\end{proof}

In the following proposition, we show that best signed approximations are principal and intermediate convergents. 
See Figure~\ref{Fig_intermediate}.

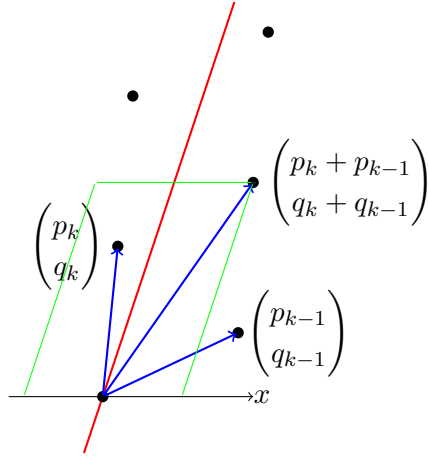
\begin{figure}
\begin{tikzpicture}[scale=0.5,inner sep=0mm]
	\draw[->] (-2.5, 0) -- (4, 0) node[right]{$x$};
	\draw[thick,red] (-.5, -1.5) -- (3.5, 10.5);
	\foreach \position in {
	(0, 0), (.4,4), (.8,8),
	(3.6, 1.7), (4,5.7), (4.4,9.7) }
    \fill \position circle (0.15cm);
	\node (O) at (0, 0) {};
	\node (P) at (3.6, 1.7) {};
	\node (P1) at (-.2, 5.7) {};
	\node (P2) at (2.1,0) {};
	\node (P3) at (-2.1,0) {};
	\node (P4) at (.4,4) {};
	\node (P5) at (4,5.7) {};
	\draw[thick, blue, ->] (O) -- (P) node[black, right, xshift=3pt]{$\begin{pmatrix} p_{k-1} \\ q_{k-1} \end{pmatrix}$};
	\draw[thick, blue, ->] (O) -- (P4) node[black, left, xshift=-5pt]{$\begin{pmatrix} p_k \\ q_k \end{pmatrix}$};
	\draw[thick, blue, ->] (O) -- (P5) node[black, right, xshift=5pt]{$\begin{pmatrix} p_k + p_{k-1} \\ q_k + q_{k-1} \end{pmatrix}$};
	\draw[thin, green] (P5) -- (P1);
	\draw[thin, green] (P5) -- (P2);
	\draw[thin, green] (P1) -- (P3);
\end{tikzpicture}
\caption{Intermediate convergents and best signed approximations of $x$.}
\label{Fig_intermediate}
\end{figure}

\begin{proposition}\label{prop_int}
A rational $p/q$ is a best signed approximation of $x\in\mathbb{R}\setminus\mathbb{Q}$ if and only if $p/q$ is a principal convergent or an intermediate convergent of $x$, i.e.,
$$
\mathscr{S}(x) = 
\left\{ \frac{p_{n,k}}{q_{n,k}} \, | \,   1 \le k < a_n, \ n \ge 1 \right\} \cup \left\{ \frac{p_n}{q_n} \, | \, n \ge 0 \right\}.
$$
\end{proposition}

\begin{proof}
Let $p/q$ be a best signed approximation of $x$.
Choose $n \ge 1$ such that $q_{n-1} \le q < q_n$.
We may assume that $n$ is odd, i.e., $q_{n-1} x - p_{n-1} > 0$ since the other case is symmetric. 
Then
$$
q_{n-2} x- p_{n-2} < qx - p <0 
\ \text{ or } \ 0 < qx - p \le q_{n-1} x- p_{n-1}.
$$
By Lemma~\ref{lem1}, we deduce that $p/q = p_{n-1}/q_{n-1}$ or $p_{n,k}/q_{n,k}$.

If $\frac pq = \frac{p_0}{q_0} = a_0$, then it is straightforward to check $p/q \in \mathscr{S}(x)$.
For $p/q = p_n/q_n$ for $n\ge 1$, by \eqref{eq:P=B}, we have $p/q \in \mathscr{B}(x) \subset \mathscr{S}(x)$.

Let $p/q = p_{n,k}/q_{n,k}$ for $n\ge 1$, $1 \le k < a_n$ and suppose that for some rational number $a/b \in \mathbb Q$, $\mathbf v_{a/b} \in P_{\mathscr B} (x,\mathbf v_{p/q}) \setminus\{\mathbf v_{p/q},-\mathbf v_{p/q}\}$.
Since $|q_{n-2} x-p_{n-2}|>|q_{n,k} x-p_{n,k}|$, by Lemma~\ref{lem1}, we have $\frac ab = \frac{p_{n-1}}{q_{n-1}}$ or $\frac ab = \frac{p_{n,i}}{q_{n,i}}$ for some $1 \le i < k$.
However, $|q_{n,i} x-p_{n,i}| > |q_{n,k} x-p_{n,k}|$ for any $ 1 \le i <  k$,
we deduce that $\frac ab = \frac{p_{n-1}}{q_{n-1}}$.
Since $q_{n-1}x - p_{n-1}$ and $q_{n,k} x-p_{n,k}$ have different signs, we conclude that
\begin{equation*}
\frac pq = \frac{p_{n,k}}{q_{n,k}} \in \mathscr{S}(x).  \qedhere
\end{equation*}
\end{proof}

By Proposition~\ref{prop_int} and \eqref{eq:P=B},
the element of $\mathscr{S}(x) \setminus \mathscr{B}(x)$ are either intermediate convergents of $x$, or $a_0$ if $a_1 = 1$. 
In the definition of $\mathscr{S}_\alpha(x)$ (see Definition~\ref{de:S_alpha}), $a/b\in\mathbb Q^{(\alpha)}$ should be the best approximation whose denominator is the biggest integer less than or equal to $q$.
In other words,
for $1 \le k < a_n$ and $n \ge 1$, we have
$$
\frac{p_{n,k}}{q_{n,k}} \in \mathscr S_\alpha(x) \qquad \text{ if } \quad \frac{p_{n-1}}{q_{n-1}} \in \mathbb Q^{(\alpha)}. $$
If $a_1 = 1$, then 
$$\frac{p_0}{q_0} = a_0 \in \mathscr{S}_{\alpha}(x) 
\quad \text{ if } \quad  \frac{p_1}{q_1} = a_0 + 1 \in \mathbb Q^{(\alpha)}.$$

Let $\{\alpha, \beta, \gamma\}=\{0,1,\infty\}$.
Combining Theorem~\ref{thm1} with Proposition~\ref{prop_int}, we describe conditions of each intermediate convergent belonging $\mathscr{B}^{(\alpha)}(x)$, or $\mathscr{B}^{(\beta,\gamma)}$.
For $\frac{p_{n-1}}{q_{n-1}} \in \mathbb Q^{(\alpha)}$ and $\frac{p_{n-2}}{q_{n-2}} \in \mathbb Q^{(\beta)}$,  by Theorem~\ref{thm1}, we have for $1 \le k < a_n$
$$
\frac{p_{n,k}}{q_{n,k}} \in \mathscr B^{(\beta,\gamma)}(x), \qquad
\frac{p_{n,k}}{q_{n,k}} \notin \mathscr B^{(\alpha, \beta)}(x), \qquad
\frac{p_{n,k}}{q_{n,k}} \notin \mathscr B^{(\alpha,\gamma)}(x),
$$
\begin{equation*}
\frac{p_{n,k}}{q_{n,k}} \in \mathscr B^{(\gamma)}(x) \ \text{ for odd $k$ \ and }  \
\frac{p_{n,k}}{q_{n,k}} \in \mathscr B^{(\beta)}(x) \ \text{ for even $k$}.
\end{equation*}
Also, if $a_1 = 1$ and $\frac{p_0}{q_0} = a_0 \in \mathbb Q^{(\alpha)}$ for $\alpha = 0$ or $\alpha = 1$, then 
$$
\frac{p_{0}}{q_{0}} \in \mathscr B^{(\alpha)}(x), \qquad
\frac{p_{0}}{q_{0}} \in \mathscr B^{(\alpha,\infty)}(x), \qquad
\frac{p_{0}}{q_{0}} \notin \mathscr B^{(0,1)}(x), \qquad
\frac{p_{0}}{q_{0}} \notin \mathscr B^{(1-\alpha,\infty)}(x).
$$

\section{A continued fraction algorithm preserving mod 2 parity}\label{sec:Delta}

In this section, we recall the coding of the geodesics on the modular surface with the regular continued fractions. 
Then, we introduce a new boundary expansion of the hyperbolic plane by using a triangle group associated with the ideal triangle whose vertices are $0$, $1$ and $\infty$.
We will show that the new expressions detect best signed approximations.

\subsection{Coding of a geodesic by the Farey Tessellation}\label{sec:Series}

The connection between continued fraction algorithm and the geodesic on the hyperbolic plane $\mathbb H = \{ z = x+yi\in \mathbb C \, | \, y > 0\}$ has been well studied.

Let $\mathcal F$ be the \emph{Farey tessellation} given by the collection of the $\mathrm{PSL}_2(\mathbb Z)$-images of the ideal triangle $\mathbf T$ with vertices at $0$, $1$, $\infty$.
Consider a hyperbolic geodesic $g$ (semicircular arc) from $y < 0$ to $x >0$.
The geodesic $g$ cuts a succession of tiles of the tessellation $\mathcal F$ starting from $\mathbf T$.
Each tile has two cutting sides by the geodesic $g$.
These cutting two sides meet in a vertex to the left, or to the right, in the oriented geodesic $g$. 
Labeling these segments of $g$ by $L$ or $R$ accordingly,
we have a sequence of $(M_i)_{i=1}^\infty$ such that $M_i \in \{ L , R \}$.
The resulting sequence $L^{a_0}R^{a_1}L^{a_2} \cdots$, 
$a_i \in \mathbb N$ for $i\ge 0$, is called the \emph{cutting sequence} of $x$.
See Figure~\ref{Fig_Series} for the Farey tessellation $\mathcal F$ and the cutting sequence along a geodesic.
Let $x > 1$ have cutting sequence $L^{a_0} R^{a_1} L^{a_2} \cdots $ with $a_i \in \mathbb N$. 
Then $x= [a_0;a_1,a_2, \dots ]$. Likewise if $0 < x < 1$ has cutting sequence $R^{a_1} L^{a_2} \cdots $ with $a_i \in \mathbb N$, then $x= [0;a_1,a_2, \dots ]$.
See \cites{Ser85, Ser91} for the details.

\begin{figure}
\begin{tikzpicture}[scale=3.8]
	\node[below] at (-1, 0) {$\frac {-1}1$};
	\node[below] at (-2/3, 0) {$\frac {-2}3$};
	\node[below] at (-1/2, 0) {$\frac {-1}2$};
	\node[below] at (-1/3, 0) {$\frac {-1}3$};
	\node[below] at (0, 0) {$\frac 01$};
	\node[below] at (1/3, 0) {$\frac 13$};
	\node[below] at (1/2, 0) {$\frac 12$};
	\node[below] at (2/3, 0) {$\frac 23$};
	\node[below] at (1, 0) {$\frac 11$};
	\node[below] at (4/3, 0) {$\frac 43$};
	\node[below] at (3/2, 0) {$\frac 32$};
	\node[below] at (5/3, 0) {$\frac 53$};
	\node[below] at (2, 0) {$\frac 21$};
	\node[below] at (7/3, 0) {$\frac 73$};
	\node[below] at (5/2, 0) {$\frac 52$};
	\node[below] at (8/3, 0) {$\frac 83$};
	\node[below] at (3, 0) {$\frac 31$};

	\node at (.5, .8) {$\mathbf T$};
	\node at (.93, 1.1) {$L$};
	\node at (1.93, .85) {$L$};
	\node at (2.33, .55) {$R$};
	\node at (2.4, .3) {$R$};
	\node at (2.32, .17) {$L$};
	
	\draw (-1, 0) -- (3, 0);
	\draw (-1,1.25) -- (-1,0);
	\draw (0,1.25) -- (0,0);
	\draw (1,1.25) -- (1,0);
	\draw (2,1.25) -- (2,0);
	\draw (3,1.25) -- (3,0);

    \draw (-1,0) arc (180:0:1/2);
    \draw (-1,0) arc (180:0:1/4);
    \draw (-1,0) arc (180:0:1/6);

    \draw (0,0) arc (0:180:1/4);
    \draw (0,0) arc (180:0:1/4);
    \draw (0,0) arc (0:180:1/6);
    \draw (0,0) arc (180:0:1/6);

    \draw (1,0) arc (0:180:1/2);
    \draw (1,0) arc (180:0:1/2);
    \draw (1,0) arc (0:180:1/4);
    \draw (1,0) arc (180:0:1/4);
    \draw (1,0) arc (0:180:1/6);
    \draw (1,0) arc (180:0:1/6);

    \draw (2,0) arc (0:180:1/4);
    \draw (2,0) arc (180:0:1/4);
    \draw (2,0) arc (0:180:1/6);
    \draw (2,0) arc (180:0:1/6);

    \draw (3,0) arc (0:180:1/2);
    \draw (3,0) arc (0:180:1/4);
    \draw (3,0) arc (0:180:1/6);
    
    \draw (-1/2,0) arc (0:180:1/12);
    \draw (-1/2,0) arc (180:0:1/12);
    \draw (1/2,0) arc (0:180:1/12);
    \draw (1/2,0) arc (180:0:1/12);
    \draw (3/2,0) arc (0:180:1/12);
    \draw (3/2,0) arc (180:0:1/12);
    \draw (5/2,0) arc (0:180:1/12);
    \draw (5/2,0) arc (180:0:1/12);

    \draw[thick,red] (2.39,0) node[below] {$x$} arc (0:180:1.22);
    
\end{tikzpicture}
\caption{The Farey tessellation $\mathcal F$ and the cutting sequence along the geodesic.}
\label{Fig_Series}
\end{figure}

We correspond $L$ and $R$ to the following matrices
$$
L = \begin{pmatrix} 1 & 1 \\ 0 & 1 \end{pmatrix} \quad \text{ and } \quad
R = \begin{pmatrix} 1 & 0 \\ 1 & 1 \end{pmatrix}.
$$
Then the sequence of matrices $M_n$ indicates the cutting sequence of the ideal triangles along the geodesic to $x$ in the Farey tessellation $\mathcal F$.
Starting with $\mathbf T$, the $m$-th ideal triangle intersecting the geodesic $g$ on the Farey tessellation $\mathcal F$ is given by $M_1 \cdots M_m \cdot \mathbf T$.

We allow $a_0$ to be a negative integer so that $x$ can be any irrational number in $\mathbb R$.
For a given $x \in \mathbb R \setminus \mathbb Q$, we have a sequence $(M_i)_{i=1}^\infty$ given by the cutting sequence $L^{a_0} R^{a_1} L^{a_2} \cdots$
with $a_0 \in \mathbb Z$ and $a_i \in \mathbb N$ for $i \ge 1$.
If $1 \le m \le |a_0|$, then
\begin{equation}\label{Farey1}
M_1 \cdots M_m 
= \begin{cases} L^{m} = \begin{pmatrix} 1 & m \\ 0 & 1 \end{pmatrix}, &\text{ if } a_0 > 0, \\
L^{-m} = \begin{pmatrix} 1 & -m \\ 0 & 1 \end{pmatrix}, &\text{ if } a_0 < 0.
\end{cases}
\end{equation}
For $m = |a_0| + a_1 + \dots + a_n + k$ with $n \ge 0$ and $0 \le k < a_{n+1}$, we have 
\begin{equation}\label{Farey2}
M_1 \cdots M_m 
= \begin{cases}
L^{a_0} R^{a_1} \cdots L^{a_n} R^{k} = \begin{pmatrix} kp_n + p_{n-1} & p_n \\ kq_n + q_{n-1} & q_n \end{pmatrix} & \text{ if $n$ is even}, \\
L^{a_0} R^{a_1} \cdots R^{a_n} L^{k}
 = \begin{pmatrix} p_n & kp_n + p_{n-1} \\ q_n & kq_n + q_{n-1} \end{pmatrix} & \text{ if $n$ is odd}. 
\end{cases}
\end{equation}

\subsection{A new expansion of real numbers by three reflections}
In this subsection, we define a new expansion of the extended real line $\mathbb{R} \cup \{ \infty\}$ which is identified as the boundary of the upper half-plane $\mathbb H$.

We consider the matrices of 
$\mathrm{PGL}_2(\mathbb Z)$ acting on $\mathbb H$ as
$$
\begin{pmatrix} a & b \\ c & d \end{pmatrix} 
\cdot z = \begin{cases}
\frac{az+b}{cz+d}, &\text{ if } ad-bc = 1, \\
\frac{a\bar z+b}{c\bar z+d}, &\text{ if }  ad-bc = -1.
\end{cases}
$$
In the calculation of the matrices, sometimes we ignore the negative sign since $A$ and $-A$ are identified in $A\in\mathrm{PGL}_2(\mathbb Z)$.
Let $\Delta$ be the triangle group generated by the three reflections $H_0, H_1, H_\infty$, where 
$$
H_0 := \begin{pmatrix} -1 & 2 \\ 0 & 1 \end{pmatrix}, \quad
H_1 := \begin{pmatrix} -1 & 0 \\ 0 & 1 \end{pmatrix} \quad \text{ and } \quad
H_\infty := \begin{pmatrix} 1 & 0 \\ 2 & -1 \end{pmatrix}.
$$
The corresponding subintervals of $H_0$, $H_1$ and $H_\infty$ on $\partial \mathbb H$ are defined by
$$
I_0:=[1,\infty], \quad I_1 := [-\infty,0] \quad \text{ and }\quad I_\infty:=[0,1].
$$
Here, we identify $\infty$ and $-\infty$.
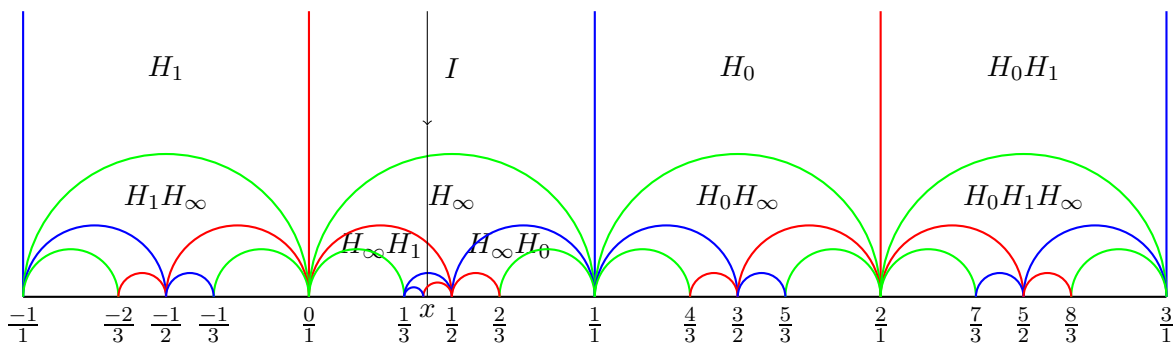
\begin{figure}
\begin{tikzpicture}[scale=3.8]
	\node[below] at (-1, 0) {$\frac {-1}1$};
	\node[below] at (-2/3, 0) {$\frac {-2}3$};
	\node[below] at (-1/2, 0) {$\frac {-1}2$};
	\node[below] at (-1/3, 0) {$\frac {-1}3$};
	\node[below] at (0, 0) {$\frac 01$};
	\node[below] at (1/3, 0) {$\frac 13$};
	\node[below] at (1/2, 0) {$\frac 12$};
	\node[below] at (2/3, 0) {$\frac 23$};
	\node[below] at (1, 0) {$\frac 11$};
	\node[below] at (4/3, 0) {$\frac 43$};
	\node[below] at (3/2, 0) {$\frac 32$};
	\node[below] at (5/3, 0) {$\frac 53$};
	\node[below] at (2, 0) {$\frac 21$};
	\node[below] at (7/3, 0) {$\frac 73$};
	\node[below] at (5/2, 0) {$\frac 52$};
	\node[below] at (8/3, 0) {$\frac 83$};
	\node[below] at (3, 0) {$\frac 31$};

	\node at (-.5, .8) {$H_1$};
	\node at (.5, .8) {$I$};
	\node at (1.5, .8) {$H_0$};
	\node at (2.5, .8) {$H_0 H_1$};
	\node at (-.5, .35) {$H_1 H_\infty$};
	\node at (.5, .35) {$H_\infty$};
	\node at (1.5, .35) {$H_0 H_\infty$};
	\node at (2.5, .35) {$H_0H_1H_\infty$};
	\node at (.3-0.05, .18) {$H_\infty H_1$};
	\node at (.7, .18) {$H_\infty H_0$};
	
	\draw[thick] (-1, 0) -- (3, 0);
	\draw[thick,blue] (-1,1) -- (-1,0);
	\draw[thick,red] (0,1) -- (0,0);
	\draw[thick,blue] (1,1) -- (1,0);
	\draw[thick,red] (2,1) -- (2,0);
	\draw[thick,blue] (3,1) -- (3,0);

    \draw[thick,green] (-1,0) arc (180:0:1/2);
    \draw[thick,blue] (-1,0) arc (180:0:1/4);
    \draw[thick,green] (-1,0) arc (180:0:1/6);

    \draw[thick,red] (0,0) arc (0:180:1/4);
    \draw[thick,red] (0,0) arc (180:0:1/4);
    \draw[thick,green] (0,0) arc (0:180:1/6);
    \draw[thick,green] (0,0) arc (180:0:1/6);

    \draw[thick,green] (1,0) arc (0:180:1/2);
    \draw[thick,green] (1,0) arc (180:0:1/2);
    \draw[thick,blue] (1,0) arc (0:180:1/4);
    \draw[thick,blue] (1,0) arc (180:0:1/4);
    \draw[thick,green] (1,0) arc (0:180:1/6);
    \draw[thick,green] (1,0) arc (180:0:1/6);

    \draw[thick,red] (2,0) arc (0:180:1/4);
    \draw[thick,red] (2,0) arc (180:0:1/4);
    \draw[thick,green] (2,0) arc (0:180:1/6);
    \draw[thick,green] (2,0) arc (180:0:1/6);

    \draw[thick,green] (3,0) arc (0:180:1/2);
    \draw[thick,blue] (3,0) arc (0:180:1/4);
    \draw[thick,green] (3,0) arc (0:180:1/6);
    
    \draw[thick,red] (-1/2,0) arc (0:180:1/12);
    \draw[thick,blue] (-1/2,0) arc (180:0:1/12);
    \draw[thick,blue] (1/2,0) arc (0:180:1/12);
    \draw[thick,red] (1/2,0) arc (180:0:1/12);
    \draw[thick,red] (3/2,0) arc (0:180:1/12);
    \draw[thick,blue] (3/2,0) arc (180:0:1/12);
    \draw[thick,blue] (5/2,0) arc (0:180:1/12);
    \draw[thick,red] (5/2,0) arc (180:0:1/12);

    \draw[->] (0.414,1) -- (0.414,0.6);
    \draw (0.414,0.6) -- (0.414,0);

    \node at (0.414,-0.05) {$x$};

    \draw[thick,red] (2/5,0) arc (180:0:1/20);
    \draw[thick,blue] (2/5,0) arc (0:180:1/30);

\end{tikzpicture}
\caption{ 
The tessellation given by the $\Delta$-images of $\mathbf T$.
Each triangle is the image of the transformation written on it.
The red lines, blue and green lines indicate the axes of the conjugates of the reflections of $H_1$, $H_0$ and $H_\infty$.
}
\label{Fig_fundamental}
\end{figure}

As in the previous section, we set $\{\alpha, \beta, \gamma\} =  \{ 0,1,\infty\}$.
Since $H_\alpha \cdot I_{\alpha} = I_\beta \cup I_\gamma,$ we have
\begin{equation*}
I_{\alpha} = H_\alpha \cdot I_\beta \cup H_\alpha \cdot I_\gamma.
\end{equation*}
Inductively, we have, for $\alpha_1 \in \{ 0,1,\infty\}$, for any $m\ge 1$, 
\begin{equation}\label{eq:I}
I_{\alpha_1} = \bigcup_{\substack{\alpha_2, \dots \alpha_m \\ \alpha_i \ne \alpha_{i-1}} } H_{\alpha_1} H_{\alpha_2} \cdots H_{\alpha_{m-1}} \cdot I_{\alpha_m}. 
\end{equation}

For an irrational number $x \in \mathbb R \setminus \mathbb Q$, we put an infinite sequence $( \alpha_i(x) )_{i=1}^\infty \in \{ 0,1,\infty \}^{\mathbb N}$ with $\alpha_i(x) \ne \alpha_{i+1}(x)$ for $i \ge 1$, given by recursively
$$
\alpha_{m} (x) = \alpha \quad \text{ if } \ \left( H_{\alpha_1} \cdots H_{\alpha_{m-1}} \right)^{-1} \cdot x =  H_{\alpha_{m-1}} \cdots H_{\alpha_1} \cdot x \in I_\alpha.
$$
Then the sequence $( \alpha_i(x) )_{i=1}^\infty$ represents a boundary point $x \in \mathbb R \cup \{ \infty \}$ of the hyperbolic space $\mathbb H$.
We call $[ \alpha_1(x), \alpha_2(x), \alpha_3(x), \cdots ]_{\Delta}$ the \emph{$\Delta$-expression} of $x$.
We obtain the $\Delta$-expression as a cutting sequence of a geodesic toward $x$ under the tessellation given by the $\Delta$-images of $\mathbf T$.
As in Figure~\ref{Fig_fundamental}, we color the $\Delta$-image of  the geodesic joining $\infty$ and $0$ red, the $\Delta$-image of the geodesic joining $\infty$ and $1$ blue, and the $\Delta$-image of the geodesic joining $0$ and $1$ green.
We denote the geodesic from $\infty$ toward $x$ by $\ell$.
We record the colors of the geodesics that $\ell$ passes through.
Then, by changing each color to the corresponding letter, we have a $\Delta$-expression of $x$, where red corresponds to $1$, blue corresponds to $0$ and green corresponds to $\infty$.
Note that
\begin{equation}\label{MHrel}
M_1 \cdots M_m \cdot \mathbf T = H_{\alpha_1} H_{\alpha_2 } \cdots H_{\alpha_m}\cdot \mathbf T.
\end{equation}

Since $I_\beta\cap I_\gamma=\{\alpha\}\not=\emptyset$ for $\{\alpha,\beta,\gamma\}=\{0,1,\infty\}$, we have  two $\Delta$-expressions of $\alpha$.
For $x=\alpha$, the first letter $\alpha_1$ can be $\beta$ or $\gamma$.
With the condition $\alpha_i(x)\not=\alpha_{i+1}(x)$, we have
\begin{equation}\label{eq:alpha}
\alpha = [\overline{\beta, \gamma} ]_\Delta := [\beta,\gamma,\beta,\gamma,\cdots]_\Delta\quad\text{ and }\quad
\alpha = [\overline{\gamma, \beta} ]_\Delta := [\gamma,\beta,\gamma,\beta,\cdots]_\Delta.
\end{equation}
If $x\in\mathbb Q^{(\alpha)}$, then there is $m$ such that $H_{\alpha_{m-1}} \cdots H_{\alpha_1}\cdot x=\alpha$.
Suppose that $m$ is minimal.
Then $\alpha_{m-1}=\alpha$, and $[\alpha_m,\alpha_{m+1},\cdots]_\Delta$ is one of the $\Delta$-expressions as in \eqref{eq:alpha}.
Except for rational points (including $\infty$), all irrational numbers have a unique expression $x = [ \alpha_1 (x), \alpha_2 (x), \alpha_3 (x), \cdots ]_{\Delta}$ with $\alpha_i \ne \alpha_{i+1}$.

For $\alpha_1, \dots, \alpha_m$ with $\alpha_i \ne \alpha_{i+1}$ for $1 \le i \le m-1$, let 
$$
I_{\alpha_1, \alpha_2, \dots, \alpha_m} :=
H_{\alpha_1} H_{\alpha_2} \cdots H_{\alpha_{m-1}} \cdot I_{\alpha_m}.$$
Then it is a cylinder set:
$$ I_{\alpha_1, \alpha_2, \dots, \alpha_m} = \{ [\delta_1, \delta_2, \dots, \delta_m,\delta_{m+1},\dots ]_\Delta \in \mathbb R \cup \{ \infty \} \, : \, \delta_i = \alpha_i \text{ for }1\le i\le m \}.$$

Let $(\alpha_i)_{i=1}^\infty \in \{ 0,1,\infty\}^{\mathbb N}$ be a sequence with $\alpha_i \ne \alpha_{i+1}$ for all $i \ge 1$. 
Choose $\beta_{m}$, $\gamma_m$ as $\{ \alpha_{m}, \beta_{m}, \gamma_m \} = \{ 0 , 1 , \infty \}$ for $m \ge 1$.
Then the two end points of $I_{\alpha_m}$ are $\beta_{m} = [\overline{\alpha_m,\gamma_{m}}]_\Delta$ and $\gamma_{m} = [\overline{\alpha_m,\beta_{m}}]_\Delta$. 
Therefore, the two end points of 
$I_{\alpha_1, \alpha_2, \dots, \alpha_m}$
are 
\begin{equation}\label{eq:ends}
\begin{split}
& H_{\alpha_1} \cdots H_{\alpha_{m-1}} \cdot \beta_{m} = [\alpha_1, \alpha_2, \dots, \alpha_{m-1}, \overline{\alpha_m , \gamma_m}]_\Delta \in \mathbb Q^{(\beta_m)}, \\
&H_{\alpha_1} \cdots H_{\alpha_{m-1}} \cdot \gamma_m = [\alpha_1, \alpha_2, \dots, \alpha_{m-1}, \overline{\alpha_{m} , \beta_{m}}]_\Delta \in \mathbb Q^{(\gamma_m)}.
\end{split}
\end{equation}

\begin{example}
Let $x = \sqrt 2 -1 = 0.4142\dots = [0; \overline{2}].$
Its $\Delta$-expression starts from $\infty, 1,0,1$, see Figure~\ref{Fig_fundamental}.
Since we have
$$
H_\infty H_1 H_0 H_1 \cdot x = \begin{pmatrix} 1 & 2 \\ 2 & 5 \end{pmatrix} \cdot x = x,
$$
its $\Delta$-expression is
$$x = [ \overline{\infty, 1,0,1}]_\Delta,$$
i.e., $\alpha_{4k+1}=\infty$, $\alpha_{4k+2}=1$, $\alpha_{4k+3}=0$, $\alpha_{4k+4}=1$ for $k\ge 0$.
We have
$$(H_\infty H_1 H_0 H_1)^k = \begin{pmatrix}P_{2k-1}&P_{2k}\\P_{2k}&P_{2k+1}\end{pmatrix},$$
where $P_n$ is a Pell number, i.e.,  $P_n = 2P_{n-1}+P_{n-2}$ with $P_0=0$ and $P_1=1$.
Thus we have
\begin{align*}
& \prod_{i=1}^{4k} H_{\alpha_i} = 
\begin{pmatrix}P_{2k-1}&P_{2k}\\P_{2k}&P_{2k+1}\end{pmatrix}
& \text{ and }  & & I_{\alpha_1,\cdots,\alpha_{4k+1}} = \Bigg[\frac{P_{2k}}{P_{2k+1}},\frac{P_{2k-1}+P_{2k}}{P_{2k}+P_{2k+1}}\Bigg],\\
& \prod_{i=1}^{4k+1} H_{\alpha_i} = 
\begin{pmatrix}P_{2k+1}&-P_{2k}\\P_{2k+2}&-P_{2k+1}\end{pmatrix}
& \text{ and }  & & I_{\alpha_1,\cdots,\alpha_{4k+2}} = \Bigg[\frac{P_{2k}}{P_{2k+1}},\frac{P_{2k+1}}{P_{2k+2}}\Bigg],\\
& \prod_{i=1}^{4k+2} H_{\alpha_i} = 
\begin{pmatrix}P_{2k+1}&P_{2k}\\P_{2k+2}&P_{2k+1}\end{pmatrix}
& \text{ and }  & & I_{\alpha_1,\cdots,\alpha_{4k+3}} = \Bigg[\frac{P_{2k}+P_{2k+1}}{P_{2k+1}+P_{2k+2}},\frac{P_{2k+1}}{P_{2k+2}}\Bigg],\\
& \prod_{i=1}^{4k+3} H_{\alpha_i} = 
\begin{pmatrix}P_{2k+1}&-P_{2k+2}\\P_{2k+2}&-P_{2k+3}\end{pmatrix}
& \text{ and }  & &
I_{\alpha_1,\cdots,\alpha_{4k+4}} = \Bigg[\frac{P_{2k+2}}{P_{2k+3}},\frac{P_{2k+1}}{P_{2k+2}}\Bigg].
\end{align*}
Further, for all $m\geq 1$, $\sqrt{2}-1\in I_{\alpha_1,\alpha_2,\dots,\alpha_m}$.
\end{example}

\subsection{Convergents of $\Delta$-expression}
A fundamental domain of the triangle group $\Delta$ is the ideal triangle $\mathbf T$ with vertices $0$, $1$ and $\infty$. 
Let 
$$
J := \begin{pmatrix} 0 & 1 \\ 1 & 0 \end{pmatrix} \quad\text{ and }\quad
K := \begin{pmatrix} -1 & 1 \\ 0 & 1 \end{pmatrix}.
$$
Then, $JKJ = KJK$.
The group
$$
\Gamma := \{ I , J, K, JK, KJ, JKJ\}
$$
forms a symmetric group of the ideal triangle $\mathbf{T}$.
Note that $\Gamma$ is isomorphic to the symmetric group of three elements.
See Figure~\ref{Fig_tri} for the figure of $\mathbf T$ and the axes of the reflections $H_\alpha$, $J$, $K$ and $JKJ$.
We regard an element $S$ of $\Gamma$ as a permutation $\sigma_S$ on $\{ 0,1,\infty\}$ defined by
\begin{equation}\label{eq:sig}
\sigma_S (\alpha) = S \cdot \alpha.
\end{equation}
Then we have the following relation
\begin{equation}\label{conjS}
H_{\sigma_S (\alpha)} = H_{S \cdot \alpha} = S H_\alpha S^{-1}.
\end{equation}

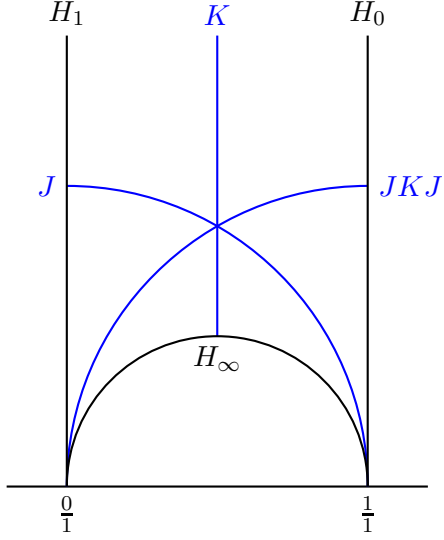
\begin{figure}
\begin{tikzpicture}[scale=4]
\node[below] at (0, 0) {$\frac 01$};
\node[below] at (1, 0) {$\frac 11$};

\draw[thick] (-.2, 0) -- (1.2, 0);
\draw[thick,blue] (1/2,1/2) -- (1/2,1.5) node[above] {$K$};
\draw[thick,blue] (0,0) arc (180:90:1) node[right] {$JKJ$};
\draw[thick,blue] (1,0) arc (0:90:1) node[left] {$J$};

\draw[thick] (0,0) -- (0,1.5) node[above] {$H_1$};
\draw[thick] (1,0) -- (1,1.5) node[above] {$H_0$};
\draw[thick] (0,0) arc (180:0:1/2);
\node[below] at (1/2, 1/2) {$H_\infty$};

\end{tikzpicture}
\caption{The fundamental domain of the triangle group $\Delta$ (surrounded by $H_1,H_0,H_\infty$) and $\langle\Delta \cup \Gamma\rangle$ (surrounded by $H_1,K, J$).}
\label{Fig_tri}
\end{figure}

Now we represent the $\Delta$-expression of $x$ according to the sequence $(M_i)_{i=1}^\infty$ defined in Section~\ref{sec:Series}. 

\begin{proposition}
Let $x$ be an irrational number with the $\Delta$-expression $[\alpha_1, \alpha_2, \dots ]_\Delta$ and the regular continued fraction $[a_0; a_1, a_2, \dots]$.
Let $(M_i)_{i=1}^\infty$ be given by the cutting sequence $L^{a_0} R^{a_1} L^{a_2} \cdots$
and let
\begin{equation}\label{eq:S}
S_i = 
\begin{cases}
K &\text{ if } \ M_i = L \ \text{ or } \ L^{-1}, \\
JKJ &\text{ if } \  M_i = R
\end{cases}
\qquad \text{and} \qquad
\eta_i = 
\begin{cases}
0 &\text{ if } \ M_i = L, \\
1 &\text{ if } \ M_i = L^{-1}, \\
\infty &\text{ if } \  M_i = R.
\end{cases}
\end{equation}
Then we have
\begin{equation}\label{alpha}
\alpha_m = S_1 \cdots S_{m-1} \cdot \eta_m
\end{equation}
and
\begin{equation}\label{MH}
M_1 \cdots M_m = H_{\alpha_1} H_{\alpha_2} \cdots H_{\alpha_m} S_1 \cdots S_{m}.
\end{equation}
\end{proposition}

\begin{proof}
We have the relations between $\{L,R\}$ and $\{H_0,H_1,H_\infty\}$ along the elements of the permutation group $\Gamma$ as follows:
\begin{equation}\label{LR}
L = H_0 K, \qquad L^{-1} = H_1 K \quad\text{ and }\quad R = H_\infty JKJ.
\end{equation}
Therefore, we have 
$$
M_i = H_{\eta_i} S_i.
$$
From the relation \eqref{eq:sig}, we have
\begin{align*}
M_1 \cdots M_m &= H_{\eta_1} S_1 H_{\eta_2} S_2 \cdots H_{\eta_m} S_m \\
&= H_{\eta_1} H_{S_1 \cdot \eta_2} H_{S_1 S_2 \cdot \eta_3} \cdots H_{S_1 \cdots S_{m-1} \cdot \eta_m} S_1 \cdots S_m.
\end{align*}
Since $S_i\cdot\mathbf T=\mathbf T$, we have 
$$M_1 \cdots M_m \cdot \mathbf T = H_{\eta_1} H_{S_1 \cdot \eta_2} H_{S_1 S_2 \cdot \eta_3} \cdots H_{S_1 \cdots S_{m-1} \cdot \eta_m}\cdot \mathbf T,$$
which is the $m$-th ideal triangle intersecting the geodesic $g$ on the Farey tessellation $\mathcal F$.
Note that $H\cdot \mathbf T=\mathbf T$ for $H\in\Delta$ implies $H=I$.
By \eqref{MHrel},
we have
$$
H_{\eta_1} H_{S_1 \cdot \eta_2} H_{S_1 S_2 \cdot \eta_3} \cdots H_{S_1 \cdots S_{m-1} \cdot \eta_m} = H_{\alpha_1} H_{\alpha_2 } \cdots H_{\alpha_m}.
$$
Since there is no relation for generators $H_0, H_1, H_\infty$ except for $H_0^2=H_1^2=H_\infty^2 = I$ in the triangle group $\Delta$, any element of $\Delta$ is uniquely expressed by $H_0$, $H_1$ and $H_\infty$.
Therefore, we deduce \eqref{alpha} and \eqref{MH}.
\end{proof}

Let $x\in\mathbb R\setminus\mathbb Q$. Let $[\alpha_1, \alpha_2, \cdots ]_\Delta$ be the $\Delta$-expression of $x$ and $[a_0; a_1, a_2, \dots]$ be the regular continued fraction expansion of $x$. 
For $m \ge |a_0 |$, we define $n \ge 0$ and $0 \le k < a_{n+1}$ such that $m = |a_0|+ a_1 + \cdots + a_n + k$.
If $m \ge |a_0| +1$, then $M_m = L$ or $R$.
By \eqref{eq:S}, we have
\begin{equation}\label{betaS}
\eta_m = S_m \cdot 1 \qquad \text{ for } \ M_m = L \ \text{ or } \  R.
\end{equation}
Therefore, using \eqref{alpha}, we have
\begin{equation}\label{alphap}
\alpha_m = S_1 \cdots S_{m-1} \cdot \eta_m = 
S_1 \cdots S_{m-1} S_m \cdot 1.
\end{equation} 
For $m\geq 1$, let $\delta_{m+1}$ be a letter in $\{0,1,\infty\}$ such that 
\begin{equation}\label{eq:del}\{ \alpha_{m}, \alpha_{m+1}, \delta_{m+1} \} = \{ 0 , 1 , \infty \}.\end{equation}
We distinguish three cases.\\
\noindent
(i) $a_0 >0$ and $1\le m \le a_0$. Then, 
by \eqref{Farey1}, \eqref{alpha} and \eqref{MH}, 
$$ H_{\alpha_1} \cdots H_{\alpha_m}  \cdot \alpha_{m+1} = \begin{cases} m, &m < a_0,\\
\infty, &m = a_0,
\end{cases}
\quad
H_{\alpha_1} \cdots H_{\alpha_m}  \cdot \delta_{m+1} = 
\begin{cases}
\infty, &m < a_0, \\
a_0, &m = a_0.
\end{cases}
$$
\noindent (ii) $a_0 <0$ and $1\le m \le |a_0|$. Then, by \eqref{Farey1}, \eqref{alpha} and \eqref{MH}, 
$$ H_{\alpha_1} \cdots H_{\alpha_m}  \cdot \alpha_{m+1} = \begin{cases}
-m+1, &m < |a_0|,\\
\infty, &m = |a_0|,
\end{cases}
\quad
H_{\alpha_1} \cdots H_{\alpha_m}  \cdot \delta_{m+1} = \begin{cases}
\infty, &m < |a_0|,\\
a_0+1, &m = |a_0|.
\end{cases}
$$
\noindent (iii) $m \ge  | a_0 | +1$. Then, by \eqref{Farey2}, \eqref{alpha} and \eqref{MH}, 
\begin{equation}\label{HHH}
H_{\alpha_1} H_{\alpha_2} \cdots H_{\alpha_m} \cdot \alpha_{m+1} = \frac{k p_n + p_{n-1}}{k q_n + q_{n-1}},
\qquad 
H_{\alpha_1} H_{\alpha_2} \cdots H_{\alpha_m} \cdot \delta_{m+1} = \frac{p_n}{q_n}.
\end{equation}

\section{Algorithms for finding best rational approximations}
\label{sec:BA}

In this section, we express best approximations, best signed approximations, best $(\alpha)$-rational approximations and best $(\alpha,\beta)$-rational approximations in terms of the $\Delta$-expression.   

\subsection{Best approximation and best signed approximation}
Let $x$ be an irrational number with the $\Delta$-expression $[\alpha_1, \alpha_2, \cdots ]_\Delta$ and the regular continued fraction $[a_0; a_1, a_2, \cdots]$.

\begin{lemma}\label{k=0}
Let $m \ge |a_0| +2$. A letter $\delta_m$ is defined as in \eqref{eq:del}.
For $m = |a_0| + a_1 + \dots + a_n + k$ with $n\ge 0$ and $0 \le k < a_{n+1}$, we have
$$\delta_m = \delta_{m+1},\text{ (i.e., }
\alpha_{m-1}=\alpha_{m+1}\text{)} \quad \text{ if and only if } \quad  k \ne 0. $$
\end{lemma}

\begin{proof}
By \eqref{alphap} and \eqref{alpha}, we have
$$\alpha_{m-1} = S_1 \cdots S_{m-1} \cdot 1 ,\qquad \alpha_{m+1} = 
S_1 \cdots S_{m-1} S_m \cdot \eta_{m+1}.
$$
Therefore, 
$$\alpha_{m-1} = \alpha_{m+1} \ \text{ if and only if } \  M_{m} = M_{m+1}.$$
 Note that $M_{m} \ne M_{m+1}$ if and only if $k = 0$. We then have
\begin{align}\label{l-equi1}
    \alpha_{m-1} = \alpha_{m+1} \ \text{ if and only if } \ k\neq 0.
\end{align}
On the other hand, since $\{\alpha_{m-1},\alpha_m,\delta_m\}=\{\alpha_m,\alpha_{m+1},\delta_{m+1}\}$, we have 
$$\{\alpha_{m-1},\delta_m\}=\{\alpha_{m+1},\delta_{m+1}\}.$$ Thus,
\begin{align}\label{l-equi2}
\alpha_{m-1} = \alpha_{m+1}
\ \text{ if and only if } \ \delta_{m} = \delta_{m+1}.
\end{align}
Combining \eqref{l-equi1} and \eqref{l-equi2} completes the proof.
\end{proof}

\begin{proposition}\label{BA}
For $x = [\alpha_1, \alpha_2, \cdots ]_\Delta $, we have 
$$
\mathscr{B}(x) = \{ [ \alpha_1, \dots, \alpha_{m-1}, \overline{ \alpha_m, \alpha_{m+1}}]_\Delta \in \mathbb Q \, | \, m \ge | a_0| +1 \}
$$
and
$$
\mathscr{S}(x) = \{ [\alpha_1, \dots, \alpha_{m-1}, \overline{ \alpha_m, \delta_{m+1}}]_\Delta \in \mathbb Q \, | \, m \ge |a_0| +1 \},
$$
where $\delta_{m+1}$ satisfies $\{ \alpha_{m}, \alpha_{m+1}, \delta_{m+1} \} = \{ 0 , 1 , \infty \}$ for $m \ge 1$.
Moreover, we have 
\begin{align*}
[\alpha_1, \dots, \alpha_{m-1}, \overline{ \alpha_m, \delta_{m+1}}]_\Delta &\in \mathscr{S}_{\delta_{m+1}} (x)
& &\text{ if } \ m = |a_0| +1 \ \text{ or } \ \alpha_{m-1} = \alpha_{m+1}, \\
[\alpha_1, \dots, \alpha_{m-1}, \overline{ \alpha_m, \delta_{m+1}}]_\Delta &\in \mathscr{B} (x)
& &\text{ if } \ m \ge |a_0| +2 \ \text{ and } \ \alpha_{m-1} \ne \alpha_{m+1}.
\end{align*}
\end{proposition}

\begin{proof}
By \eqref{eq:alpha}, we have
$$\delta_{m+1} =[\overline{\alpha_{m}, \alpha_{m+1}}]_\Delta,  
\qquad \alpha_{m+1} = [\overline{\alpha_{m}, \delta_{m+1}}]_\Delta.$$
Therefore, we have
$$
[ \alpha_1, \alpha_2, \dots, \alpha_{m-1}, \overline{\alpha_{m}, \alpha_{m+1}}]_\Delta = H_{\alpha_1} \cdots H_{\alpha_{m-1}} \cdot \delta_{m+1} = H_{\alpha_1} \cdots H_{\alpha_{m-1}} H_{\alpha_{m}} \cdot \delta_{m+1},$$
$$
[ \alpha_1, \alpha_2, \dots, \alpha_{m-1}, \overline{\alpha_{m}, \delta_{m+1}}]_\Delta= H_{\alpha_1} \cdots H_{\alpha_{m-1}} \cdot  \alpha_{m+1} = H_{\alpha_1} \cdots H_{\alpha_{m-1}} H_{\alpha_{m}} \cdot \alpha_{m+1}.$$
Let $m\ge |a_0|+1$.
By \eqref{HHH}, we have
$$
[ \alpha_1, \dots, \alpha_{m-1}, \overline{ \alpha_m, \alpha_{m+1}}]_\Delta = \frac{p_n}{q_n}\quad \text{ and }\quad
[ \alpha_1, \dots, \alpha_{m-1}, \overline{ \alpha_m, \delta_{m+1}}]_\Delta = \frac{kp_n+p_{n-1}}{kq_n+ q_{n-1}},$$
where $n \ge 0$ and $0 \le k < a_{n+1}$ such that $m = |a_0|+ a_1 + \cdots + a_n + k$.
Note that if $a_1 = 1$, then the condition $m \ge |a_0| +1$ implies $n \ge 1$. For other cases, $n$ ranges from $0$. 
Therefore, \eqref{eq:P=B} implies that 
$[\alpha_1, \dots, \alpha_{m-1}, \overline{ \alpha_m, \alpha_{m+1}}]_\Delta$ with $m \ge |a_0| +1$ consists of $\mathscr{B}(x)$.
The condition $m\ge |a_0| +1$ yields that $1 \le k < a_{n+1}$ for $n = 0$ or $0 \le k < a_{n+1}$ for $n \ge 1$. 
By Proposition~\ref{prop_int}, we have 
$\{ [\alpha_1, \dots, \alpha_{m-1}, \overline{ \alpha_m, \delta_{m+1}}]_\Delta \, | \, m \ge |a_0| +1 \} = \mathscr{S}(x) $.

Suppose that $m = |a_0| +1$. 
If $a_1 = 1$, then $m= |a_0| + a_1$ with $n=1$ and $k=0$. Thus,
\begin{align*}
[\alpha_1, \dots, \alpha_{m-1}, \overline{ \alpha_m, \delta_{m+1}}]_\Delta &= \frac{p_0}{q_0} = a_0 \notin\mathscr{B}(x), \\
[\alpha_1, \dots, \alpha_{m-1}, \overline{ \alpha_m, \alpha_{m+1}}]_\Delta &= \frac{p_1}{q_1} = a_0+1 \in\mathscr{B}(x) \cap \mathbb Q^{(\delta_{m+1})}. 
\end{align*}
If $a_1 \ge 2$, then $m = |a_0| + 1$ with $n = 0$ and $k=1$. Thus,
\begin{align*}
[\alpha_1, \dots, \alpha_{m-1}, \overline{ \alpha_m, \delta_{m+1}}]_\Delta &= \frac{p_0+p_{-1}}{q_0+q_{-1}} = a_0+1\notin\mathscr{B}(x), \\
[\alpha_1, \dots, \alpha_{m-1}, \overline{ \alpha_m, \alpha_{m+1}}]_\Delta &= \frac{p_0}{q_0} = a_0\in\mathscr{B}(x) \cap \mathbb Q^{(\delta_{m+1})}.\end{align*}

Suppose that $m \ge |a_0| +2$.
If $\alpha_{m-1} = \alpha_{m+1}$, 
then by Lemma~\ref{k=0}, we have $k \ne 0$. Hence,
\begin{align*}
[\alpha_1, \dots, \alpha_{m-1}, \overline{ \alpha_m, \delta_{m+1}}]_\Delta &= \frac{kp_n + p_{n-1}}{kq_n+q_{n-1}}\notin\mathscr{B}(x), \\
[\alpha_1, \dots, \alpha_{m-1}, \overline{ \alpha_m, \alpha_{m+1}}]_\Delta &= \frac{p_n}{q_n}\in\mathscr{B}(x) \cap \mathbb Q^{(\delta_{m+1})}. \end{align*}
If $\alpha_{m-1} \ne \alpha_{m+1}$, 
then by Lemma~\ref{k=0}, we have $k = 0$. Hence,
\begin{equation*}
[\alpha_1, \dots, \alpha_{m-1}, \overline{ \alpha_m, \delta_{m+1}}]_\Delta = \frac{p_{n-1}}{q_{n-1}} \in\mathscr{B}(x). \qedhere
\end{equation*}
\end{proof}

\begin{example}\label{ex:bs}
For $x = \sqrt 2 -1 = [\overline{\infty, 1 , 0 ,1}]_\Delta$, we have
$$
\mathscr{B}(x) = \left \{ [\overline{\infty,1}]_\Delta = 0, \,
[\infty,\overline{1,0}]_\Delta = \frac 12, \,
[\infty,1,0,\overline{1,\infty}]_\Delta = \frac 25, \,
[\infty,1,0,1,\infty,\overline{1,0}]_\Delta = \frac 5{12}, \,
 \dots \right \}.
$$
Also, we have
\begin{gather*}
\mathscr{S}_0 (x) = \left \{ [\overline{\infty,0}]_\Delta = 1, \,
[\infty, 1,0,1,\overline{\infty, 0}]_\Delta = \frac 37, \,
[\infty, 1,0,1,\infty,1,0,1,\overline{\infty, 0}]_\Delta = \frac{17}{41}, \,
 \dots \right \}, \\
\mathscr{S}_\infty(x) = \left \{ [\infty,1,\overline{0,\infty}]_\Delta = \frac 13, \,
[\infty,1,0,1,\infty,1,\overline{0,\infty}]_\Delta = \frac{7}{17}, \,
 \dots \right \}, \qquad \mathscr{S}_1(x) = \emptyset.
\end{gather*}
\end{example}

\subsection{Best ($\alpha$) and ($\alpha, \beta$)-rational approximations} 
Since 
$$[ \alpha_1, \dots, \alpha_{m-1}, \overline{ \alpha_m, \delta_{m+1}}]_\Delta \in \mathbb Q^{(\alpha)}\ \text{ if and only if } \ \alpha_{m+1} = \alpha,$$ 
applying Theorem~\ref{thm1} to Proposition~\ref{BA}, we have the following proposition. 

\begin{proposition}\label{aBA}
For $x = [\alpha_1, \alpha_2, \cdots ]_\Delta $ and $\alpha \in \{ 0,1,\infty\}$, we have 
$$
\mathscr{B}^{(\alpha)}(x) = \{ [ \alpha_1, \dots, \alpha_{m-1}, \overline{ \alpha_m, \delta_{m+1}}]_\Delta \in \mathbb Q \, | \, \alpha_{m+1} = \alpha, \ m \ge |a_0| +1 \},
$$
where $\delta_{m+1}$ satisfies $\{ \alpha_{m}, \alpha_{m+1}, \delta_{m+1} \} = \{ 0 , 1 , \infty \}$ for $m \ge 1$.
\end{proposition}

\begin{example}
For $x = \sqrt 2 -1 = [\overline{\infty, 1 , 0 ,1}]_\Delta$, we have
\begin{align*}
\mathscr{B}^{(0)} (x) &= \left \{ 
[\overline{\infty,1}]_\Delta = 0, \,
[\infty, 1,0,\overline{1,\infty}]_\Delta = \frac 25, \,
[\infty, 1,0,1,\infty,1,0,\overline{1,\infty}]_\Delta = \frac{12}{29}, \,
 \dots \right \}, \\
\mathscr{B}^{(1)}(x) &= \left \{ [\overline{\infty,0}]_\Delta = 1, \,
[\infty,1,\overline{0,\infty}]_\Delta = \frac 13, \,
[\infty,1,0,1,\overline{\infty,0}]_\Delta = \frac 37, \,
 \dots \right \}, \\
 \mathscr{B}^{(\infty)}(x) &= \left \{ [\infty,1,0,\overline{1,0}]_\Delta = [\infty,\overline{1,0}]_\Delta = \frac 12, \,
[\infty, 1,0,1,\infty, \overline{1,0}]_\Delta = \frac 5{12}, \,
 \dots \right \}.
\end{align*}
\end{example}

Finally, we complete the proof of Theorem~\ref{thm2} by the following proposition.

\begin{proposition}\label{abBA}
Let $x = [ \alpha_1, \alpha_2, \alpha_3, \cdots ]_\Delta$ be an irrational number with regular continued fraction expansion $[a_0 ; a_1, \dots]$ and $\{\alpha, \beta, \gamma\} = \{ 0,1,\infty\}$.
Then, we have
\begin{equation}\label{eq:abBA}
\mathscr{B}^{(\alpha,\beta)}(x) = \big\{ [ \alpha_1, \dots, \alpha_{m-1}, \overline{ \alpha_m, \gamma }]_\Delta \in \mathbb Q \, | \, \alpha_m \ne \gamma, \ m \ge |a_0| +1 \big\}.
\end{equation}
\end{proposition}

\begin{proof}
Let $\mathscr{C}$ be the set on the right-hand side of \eqref{eq:abBA}. 
To show $\mathscr{B}^{(\alpha,\beta)}= \mathscr{C}$, we apply Theorem~\ref{thm1} (ii) which states that   $\mathscr{B}^{(\alpha,\beta)}(x) = \left(  \mathscr{B}(x) \cap \mathbb Q^{(\alpha,\beta)} \right) \cup \mathscr{S}_\gamma(x)$. 
We first show $\mathscr{B}(x)\cap \mathbb Q^{(\alpha,\beta)}\subset \mathscr{C}$.
Let $[ \alpha_1, \dots, \alpha_{m-1}, \overline{ \alpha_m, \alpha_{m+1}}]_\Delta \in \mathscr{B}(x) \cap \mathbb Q^{(\alpha,\beta)}$ with $m \ge |a_0| +1$. 
Then, we have $[ \overline{ \alpha_m, \alpha_{m+1}}]_\Delta = \alpha$ or $\beta$, 
which is followed that either $\alpha_m = \gamma$ or $\alpha_{m+1} = \gamma$.
If $\alpha_{m} = \gamma$, then $\alpha_{m+1} \ne \gamma$ and
$[ \alpha_1, \dots, \alpha_{m-1}, \overline{ \alpha_m, \alpha_{m+1}}]_\Delta = [ \alpha_1, \dots, \alpha_m, \overline{ \alpha_{m+1}, \gamma}]_\Delta$.
Therefore, we have $\mathscr{B}(x)\cap \mathbb Q^{(\alpha,\beta)}\subset \mathscr{C}$.
On the other hand,  $\mathscr{S}_\gamma(x)\subset \mathscr{C}$ is a direct consequence of Proposition~\ref{BA}.
Hence, we have $\mathscr{B}^{(\alpha,\beta)}\subset  \mathscr{C}$.

Now, it is left to show $\mathscr{C}\subset \mathscr{B}^{(\alpha,\beta)}$.
Suppose that $\alpha_m \ne \gamma$ and $m \ge |a_0|+1$. 
If $\alpha_{m+1} = \gamma$, then 
$$[ \alpha_1, \dots, \alpha_{m-1}, \overline{ \alpha_m, \gamma }]_\Delta \in \mathscr{B}(x) \cap \mathbb Q^{(\alpha,\beta)}\subset \mathscr{B}^{(\alpha,\beta)}(x).
$$
If $\alpha_{m+1} \ne \gamma$, $\alpha_{m-1} = \gamma$ and $m \ge |a_0|+2$, then 
$$[ \alpha_1, \dots, \alpha_{m-1}, \overline{ \alpha_m, \gamma }]_\Delta = [ \alpha_1, \dots, \alpha_{m-2}, \overline{ \gamma, \alpha_m }]_\Delta \in \mathscr{B}(x) \cap \mathbb Q^{(\alpha,\beta)}\subset \mathscr{B}^{(\alpha,\beta)}(x).
$$
If $\alpha_{m+1} \ne \gamma$, $\alpha_{m-1} \ne \gamma$ and $m \ge |a_0|+2$, then 
$\alpha_{m-1} = \alpha_{m+1}$. 
Therefore, by Proposition~\ref{BA}, we have
$$[ \alpha_1, \dots, \alpha_{m-1}, \overline{ \alpha_m, \gamma }]_\Delta \in \mathscr{S}_\gamma(x)\subset \mathscr{B}^{(\alpha,\beta)}(x).
$$
If $\alpha_{m+1} \ne \gamma$ and $m = |a_0| +1$, then $\delta_{m+1} = \gamma$. Thus by Proposition~\ref{BA}, we have
\begin{equation*}
[ \alpha_1, \dots, \alpha_{m-1}, \overline{ \alpha_m, \gamma }]_\Delta \in \mathscr{S}_\gamma(x) \subset \mathscr{B}^{(\alpha,\beta)}(x). \qedhere
\end{equation*}
\end{proof}

\begin{example}\label{ex:ab}
For $x = \sqrt 2 -1 = [\overline{\infty, 1 , 0 ,1}]_\Delta$, we have
\begin{align*}
\mathscr{B}^{(0,1)} (x) &= \left \{ [\overline{\infty,1}]_\Delta = 0, \,[\infty,1,\overline{0,\infty}]_\Delta = \frac 13, \,
[\infty,1,0,\overline{1,\infty}]_\Delta = \frac 25, \,
 \dots \right \}, \\
\mathscr{B}^{(0,\infty)}(x) &= \left \{  
[\overline{\infty,1}]_\Delta = 0, \,
[\infty,\overline{1,0}]_\Delta = \frac 12, \,
[\infty,1,0,\overline{1,\infty}]_\Delta = \frac 25, \,
\dots \right \}, \\
\mathscr{B}^{(1,\infty)}(x) &= \left \{ 
 [\overline{\infty,0}]_\Delta = 1, \,
[\infty,\overline{1,0}]_\Delta = \frac 12, \,
[\infty, 1,0,1,\overline{\infty, 0}]_\Delta = \frac 37, \,
 \dots \right \}.
\end{align*}
\end{example}

\section{Continued fraction maps and $\Delta$-expression}\label{sec:maps}
The continued fraction map is an interval map that plays a role in the left shift map of the continued fraction expansions.
In this section, we reconstruct various continued fraction maps using the $\Delta$-expression.
The three reflections $H_0, H_1, H_\infty$ induces semi-regular continued fractions of the form 
$$
a_0 + \cfrac{\epsilon_0}{a_1 + \cfrac{\epsilon_1}{a_2 + \cfrac{\epsilon_2}{\ddots}}},
$$
where $a_0 \in \mathbb Z$, $a_i \in \mathbb N$ for $i \ge 1$ and $\epsilon_i \in \{ -1, 1\}$ for $i \ge 0$ with the condition $a_i+\epsilon_i\ge 1$ for all $i\ge 1$, and $a_i+\epsilon_i\ge 2$ infinitely often (see \cites{DK00}).  
We elaborate on the relationship between several semi-regular continued fraction algorithms and the $\Delta$-expression.

\subsection{Farey map and Gauss map} 
The Gauss map given by $\lfloor \frac1x\rfloor$ for $x\in (0,1]$, is the continued fraction map of the regular continued fraction.
The Farey map defined by $\frac{x}{1-x}$ for $[0,\frac 12]$, and $\frac{1-x}{x}$ for $[\frac 12,1]$, is a slow-down map of the Gauss map.

We first represent the Farey map in terms of the $\Delta$-expression.
Define $\varphi : I_\infty \to I_\infty $, where  $I_\infty= [0,1]$, by  
\begin{equation*}
\varphi ([ \infty, \alpha_2, \alpha_3, \cdots ]_\Delta) =
[ \sigma (\alpha_{2}), \sigma (\alpha_{3}), \sigma(\alpha_{4}), \cdots ]_\Delta,
\end{equation*}
where $\sigma$ is the permutation on $\{ 0,1, \infty \}$ satisfying that 
$\sigma(\infty) = 1$, $\sigma(\alpha_2) = \infty$.
Thus, we have, for $x = [ \infty, \alpha_2 , \alpha_3 , \cdots ]_\Delta$,
$$\varphi (x) = \begin{cases}
[\sigma_{JKJ}(\alpha_2) , \sigma_{JKJ}(\alpha_3) , \cdots ]_\Delta = JKJH_\infty \cdot [ \infty, \alpha_2 , \alpha_3 , \cdots ]_\Delta , & \text{ if } \alpha_2 = 1, \\
[\sigma_{KJ}(\alpha_2) , \sigma_{KJ}(\alpha_3) , \cdots ]_\Delta = KJH_\infty \cdot [ \infty, \alpha_2 , \alpha_3 , \cdots ]_\Delta , &\text{ if } \alpha_2 = 0.
\end{cases} 
$$ 
Thus we derive that $\varphi$ is the Farey map:
$$\varphi (x) = \begin{cases}
JKJH_\infty \cdot x = \frac{x}{1-x}, & \text{ if } \ 0 \le x < \frac 12, \\
KJH_\infty \cdot x = \frac{1-x}{x}, & \text{ if } \ \frac 12\le x \le 1.
\end{cases} 
$$ 

By the fact that the Gauss map is an acceleration of the Farey map, we represent the Gauss map in terms of the $\Delta$-expression.
Suppose that $x = [ \infty, \alpha_2, \alpha_3, \cdots ]_\Delta\in I_\infty$ has the regular continued fraction $[0; a_1, a_2, \dots ]$.
We define $m = m(x) = \min\{ j \ge 1  | \, \alpha_{j+1} = 0\}$. 
Then $m = a_1$.
Define $\psi : I_\infty \to I_\infty $ by 
\begin{equation*}
\psi ( [ \infty, \alpha_2 , \alpha_3 , \cdots ]_\Delta  ) =
[ \sigma (\alpha_{m+1}), \sigma (\alpha_{m+2}), \sigma(\alpha_{m+3}), \cdots ]_\Delta,
\end{equation*}
where $\sigma$ is the permutation on $\{ 0,1, \infty \}$ satisfying that 
$\sigma(\alpha_{m}) = 1$, $\sigma(\alpha_{m+1}) = \infty$.
If $m$ is even, then $\alpha_{m+1}=0$, $\alpha_m=1$, and $x= [(\infty,1)^{m/2},0,\alpha_{m+2},\dots]\in(H_\infty H_1)^{m/2}\cdot I_0 = [\frac 1{m+1},\frac1m]$.
If $m$ is odd, then $\alpha_{m+1}=0$, $\alpha_m=\infty$, and $x= [(\infty,1)^{(m-1)/2},0,\alpha_{m+2},\dots]\in(H_\infty H_1)^{(m-1)/2}H_\infty\cdot I_0 = [\frac 1{m+1},\frac1m]$.
Thus, $\psi$ is the Gauss map, since 
$$
\psi(x)  = \begin{cases}
J \cdot [ 0, \alpha_{m+2}, \alpha_{m+3}, \cdots ]_\Delta =
J (H_1 H_\infty)^{m/2} \cdot x = \frac 1x - m, &\text{ if $m$ is even},\\
KJ \cdot [ 0, \alpha_{m+2}, \alpha_{m+3}, \cdots ]_\Delta =
KJ H_\infty (H_1 H_\infty)^{(m-1)/2} \cdot x = \frac 1x - m, &\text{ if $m$ is odd}.
\end{cases}
$$ 

In the definition of $\psi$, if we change the condition of $\sigma$ slightly, then we have another continued fraction map which is called \emph{the by-excess continued fraction map}.
The maps is one of $\alpha$-continued fraction with $\alpha=0$ and it is defined by $\lceil \frac 1x \rceil-\frac1x$ for $x\in(0,1]$ (see e.g. \cite{LMNN10}).
Let $\psi_- : I_\infty = [0,1] \to I_\infty$ be given by  
\begin{equation*}
\psi_- ( [ \infty, \alpha_2 , \alpha_3 , \cdots ]_\Delta  ) =
[ \sigma_- (\alpha_{m+1}), \sigma_- (\alpha_{m+2}), \sigma_-(\alpha_{m+3}), \cdots ]_\Delta,
\end{equation*}
where $\sigma_-$ is the permutation on $\{ 0,1, \infty \}$ satisfying that 
$\sigma_-(\alpha_{m}) = 0$, $\sigma_-(\alpha_{m+1}) = \infty$.
Then, we verify that $\psi_-$ is the by-excess continued fraction map by
$$
\psi_-(x)  = \begin{cases}
KJ \cdot [ 0, \alpha_{m+2}, \alpha_{m+3}, \cdots ]_\Delta =
KJ (H_1 H_\infty)^{m/2} \cdot x = m+1 - \frac 1x, &\text{ if $m$ is even},\\
J \cdot [ 0, \alpha_{m+2}, \alpha_{m+3}, \cdots ]_\Delta =
J H_\infty (H_1 H_\infty)^{(m-1)/2} \cdot x = m+1 - \frac 1x, &\text{ if $m$ is odd}.
\end{cases}
$$ 

\subsection{Even and odd continued fraction maps}

The even continued fraction map and the odd continued fraction map are the maps defined on the unit interval by the distance from $\frac 1x$ to its nearest even integer, or its nearest odd integer, respectively.
We represent these maps by using the permutations $J$ and $KJ$ as follows.

Let $x = [ \infty, \alpha_2, \alpha_3, \cdots ]_\Delta \in I_\infty$ and $m = m(x) = \min \{ j \ge 1 \, | \, \alpha_{j+1} = 0 \}$.
Define a map $\psi_J :I_\infty\to I_{\infty}$ by
$$
\psi_J (x) = [\sigma_J (0), \sigma_J (\alpha_{m+2}), \cdots ]_\Delta
= [\infty ,\sigma_J (\alpha_{m+2}), \cdots ]_\Delta.
$$
Then, $\psi_J$ is the even continued fraction map, since
$$
\psi_J(x) = J \cdot x = \begin{cases}
J (H_1 H_\infty)^{m/2} \cdot x = \frac 1x -m, &\text{ if $m$ is even},\\
J H_\infty (H_1 H_\infty)^{(m-1)/2} \cdot x = m+1 - \frac 1x, &\text{ if $m$ is odd}.
\end{cases}
$$ 

Define a map $\psi_{KJ} :I_\infty\to I_{\infty}$ by
$$
\psi_{KJ} (x) = [\sigma_{KJ} (0), \sigma_{KJ} (\alpha_{m+2}), \cdots ]_\Delta
= [\infty ,\sigma_{KJ} (\alpha_{m+2}), \cdots ]_\Delta
$$
Then, $\psi_{KJ}$ is the odd continued fraction map, since
$$
\psi_{KJ}(x) = KJ \cdot x =
\begin{cases}
KJ (H_1 H_\infty)^{m/2} \cdot x = m+1 - \frac 1x, &\text{ if $m$ is even},\\
KJ H_\infty (H_1 H_\infty)^{(m-1)/2} \cdot x = \frac 1x -m, &\text{ if $m$ is odd}.
\end{cases}
$$ 
Moreover, we can verify that the convergent (or finite truncation) of the even continued fraction is a best $(0,\infty)$-rational approximation, and vice versa, which is proven by Short and Walker \cite{SW14}.

In fact, let $(m_i)_{i=0}^\infty$ be a sequence defined by
$$
m_{i} = \min \{ m \ge m_{i-1} \, | \, \alpha_{m+1} = \infty \ \text{ for even } i \ \text{ and }  \alpha_{m+1} = 0 \ \text{ for odd } i \} \  \text{ with } \ m_0 = 0.
$$
Let $m_i +1 \le m \le m_{i+1}$ such that $\alpha_{m} \ne 1$.
Then, $\alpha_m = \alpha_{m_i+1} = J^i \cdot \infty$. 
Thus, we have
\begin{equation}\label{eq}
\begin{split}
[\alpha_1, \dots , \alpha_{m-1} , \overline{\alpha_m, 1}] 
&= [\alpha_{1}, \dots , \alpha_{m_i}, \overline{\alpha_{m_i+1}, 1}] \\
&= H_{\alpha_1} \cdots H_{\alpha_{m_i}} \cdot [\overline{\alpha_{m_i+1}, 1}] 
= H_{\alpha_1} \cdots H_{\alpha_{m_i}} J^i \cdot 0.
\end{split}
\end{equation}
For an irrational number $x \in I_\infty$, let $\left(\psi_J\right)_{x}^{-i}$ be the inverse map of $(\psi_J)^i$ on the maximal closed interval $I_x$ containing $x$ on which $(\psi_J)^i$ is continuous.
Then, by
$$
(\psi_J)^i (x) = [J^i \cdot \alpha_{m_i+1}, J^i \cdot \alpha_{m_i+2}, \cdots ]_\Delta, 
$$
we have
\begin{equation*}
\left(\psi_J\right)_{x}^{-i} (0) = H_{\alpha_1} H_{\alpha_2} \cdots, H_{\alpha_{m_i}} J^i \cdot 0
= [\alpha_{1}, \dots , \alpha_{m_i-1}, \overline{\alpha_{m_i}, 1}]. 
\end{equation*}
Therefore, by \eqref{eq}, we have
$$\mathscr{B}^{(0,\infty)}(x) = \{ \left(\psi_J\right)_{x}^{-i} (0) \, | \, i \ge 1\}.$$
Hence, the even continued fraction gives  best $(0,\infty)$-rational approximations.

\subsection{Odd-odd continued fraction}
The authors explored a continued fraction algorithm generating the best $(1)$-rational approximations, which is named \emph{the odd-odd continued fracion} \cite{KLL22}. The odd-odd continued fraction map is defined by $\frac{kx-(k-1)}{k-(k+1)x}$ if $x\in[\frac{k-1}{k},\frac{2k-1}{2k+1}]$, and $\frac{k-(k+1)x}{kx-(k-1)}$ if $x\in[\frac{2k-1}{2k+1},\frac{k}{k+1}]$ for $k\ge1$.

Let $x = [ \infty, \alpha_2, \alpha_3, \cdots ]_\Delta \in I_\infty$ be a real number 
and $\tilde m = \tilde m (x) = \min \{ j \ge 1 \, | \, \alpha_{j} = 1\}$.
Define a map $\psi_1 :I_\infty\to I_{\infty}$ by
$$
\psi_1 (x) = \begin{cases} [\alpha_{\tilde m+1}, \alpha_{\tilde m+2}, \cdots ]_\Delta, &\text{ if } \alpha_{\tilde m+1} = \infty, \\
[ \sigma_J( \alpha_{\tilde m+1}), \sigma_J (\alpha_{\tilde m+2}), \cdots ]_\Delta, &\text{ if } \alpha_{\tilde m+1} = 0.
\end{cases} 
$$
Then, by a similar argument as in the previous subsections, we can check that $\psi_1$ is the odd-odd continued fraction map.

For an irrational number $x \in I_\infty$, let $\left(\psi_1\right)^{-i}_{x}$ be the inverse map of $(\psi_1)^i$ on the maximal closed interval $I_x$ containing $x$ on which $(\psi_1)^i$ is continuous,
and let $(\tilde m_i)_{i=1}^\infty$ be a sequence defined by $\alpha_{\tilde m_i} = 1$.
Then, we have
\begin{align*}
\left(\psi_1^{-i}\right)_{x} (1) &= H_{\alpha_1} H_{\alpha_2} \cdots H_{\alpha_{\tilde m_i}} \cdot 1
= [\alpha_1, \alpha_2, \dots, \alpha_{\tilde m_i-1}, 1, \overline{\alpha_{\tilde m_{i}+1}, \delta_{\tilde m_i+1} }] \\  
&= [\alpha_1, \alpha_2, \dots, \alpha_{\tilde m_i-1}, 1,  \alpha_{\tilde m_i+1}, \dots, \alpha_{\tilde m_{i+1}-2}, \overline{\alpha_{\tilde m_{i+1}-1}, \delta_{\tilde m_{i+1}} }],
\end{align*}
where $\delta_{\tilde m_i+1}$ is defined as in \eqref{eq:del}.
Therefore, we reconfirm that the best $(1)$-rational approximation is generated by $\psi_1$:
$$\mathscr{B}^{(1)}(x) = \{ \left(\psi_1^{-i}\right)_{x} (1) \, | \, i \ge 1\}.$$

\section*{Acknowledgement}

D.K.~was supported by the National Research Foundation of Korea (NRF-2018R1A2B6001624, RS-2023-00245719) and the Dongguk University Research Fund of 2022.
S.L. was supported by the Institute for Basic Science (IBS-R003-D1).


\begin{bibdiv}
\begin{biblist}


\bib{BM18}{article}{
   author={Boca, Florin P.},
   author={Merriman, Claire},
   title={Coding of geodesics on some modular surfaces and applications to
   odd and even continued fractions},
   journal={Indag. Math. (N.S.)},
   volume={29},
   date={2018},
   number={5},
   pages={1214--1234},
   issn={0019-3577},
   review={\MR{3853422}},
   doi={10.1016/j.indag.2018.05.004},
}

\bib{DK00}{article}{
   author={Dajani, Karma},
   author={Kraaikamp, Cor},
   title={``The mother of all continued fractions''},
   note={Dedicated to the memory of Anzelm Iwanik},
   journal={Colloq. Math.},
   volume={84/85},
   date={2000},
   pages={109--123},
   issn={0010-1354},
   review={\MR{1778844}},
   doi={10.4064/cm-84/85-1-109-123},
}

\bib{EW11}{book}{
   author={Einsiedler, Manfred},
   author={Ward, Thomas},
   title={Ergodic theory with a view towards number theory},
   series={Graduate Texts in Mathematics},
   volume={259},
   publisher={Springer-Verlag London, Ltd., London},
   date={2011},
   pages={xviii+481},
   isbn={978-0-85729-020-5},
   review={\MR{2723325}},
   doi={10.1007/978-0-85729-021-2},
}






\bib{Khi64}{book}{
   author={Khinchin, A. Ya.},
   title={Continued fractions},
   edition={Translated from the third (1961) Russian edition},
   note={With a preface by B. V. Gnedenko;
   Reprint of the 1964 translation},
   publisher={Dover Publications, Inc., Mineola, NY},
   date={1997},
   pages={xii+95},
   isbn={0-486-69630-8},
   review={\MR{1451873}},
}

%

\bib{KLL22}{article}{
   author={Kim, Dong Han},
   author={Lee, Seul Bee},
   author={Liao, Lingmin},
   title={Odd-odd continued fraction algorithm},
   journal={Monatsh. Math.},
   volume={198},
   date={2022},
   number={2},
   pages={323--344},
   issn={0026-9255},
   review={\MR{4421912}},
   doi={10.1007/s00605-022-01704-2},
}


\bib{KL96}{article}{
   author={Kraaikamp, Cornelis},
   author={Lopes, Artur},
   title={The theta group and the continued fraction expansion with even
   partial quotients},
   journal={Geom. Dedicata},
   volume={59},
   date={1996},
   number={3},
   pages={293--333},
   issn={0046-5755},
   review={\MR{1371228}},
   doi={10.1007/BF00181695},
}

%

\bib{LMNN10}{article}{
   author={Luzzi, Laura},
   author={Marmi, Stefano},
   author={Nakada, Hitoshi},
   author={Natsui, Rie},
   title={Generalized Brjuno functions associated to $\alpha$-continued
   fractions},
   journal={J. Approx. Theory},
   volume={162},
   date={2010},
   number={1},
   pages={24--41},
   issn={0021-9045},
   review={\MR{2565823}},
   doi={10.1016/j.jat.2009.02.004},
}




\bib{Sch82}{article}{
   author={Schweiger, Fritz},
   title={Continued fractions with odd and even partial quotients},
   journal={Arbeitsber. Math. Inst. Univ. Salzburg},
   volume={4},
   date={1982},
   pages={59--70},
}

\bib{Sch84}{article}{
   author={Schweiger, Fritz},
   title={On the approximation by continued fractions with odd and even partial quotients},
   journal={Arbeitsber. Math. Inst. Univ. Salzburg},
   volume={1},
   date={1984},
   number={2},
   pages={105--114},
}


\bib{Ser85}{article}{
   author={Series, Caroline},
   title={The modular surface and continued fractions},
   journal={J. London Math. Soc. (2)},
   volume={31},
   date={1985},
   number={1},
   pages={69--80},
   issn={0024-6107},
   review={\MR{810563}},
   doi={10.1112/jlms/s2-31.1.69},
}

\bib{Ser91}{article}{
   author={Series, Caroline},
   title={Geometrical methods of symbolic coding},
   conference={
      title={Ergodic theory, symbolic dynamics, and hyperbolic spaces},
      address={Trieste},
      date={1989},
   },
   book={
      series={Oxford Sci. Publ.},
      publisher={Oxford Univ. Press, New York},
   },
   isbn={0-19-853390-X},
   isbn={0-19-859685-5},
   date={1991},
   pages={125--151},
   review={\MR{1130175}},
}

\bib{SW14}{article}{
   author={Short, Ian},
   author={Walker, Mairi},
   title={Even-integer continued fractions and the Farey tree},
   conference={
      title={Symmetries in graphs, maps, and polytopes},
   },
   book={
      series={Springer Proc. Math. Stat.},
      volume={159},
      publisher={Springer, [Cham]},
   },
   date={2016},
   pages={287--300},
   review={\MR{3516227}},
   doi={10.1007/978-3-319-30451-9{\_}15},
}


\end{biblist}
\end{bibdiv}
\end{document}